\documentclass[12pt]{article}

\usepackage[latin1]{inputenc}
\usepackage{fullpage}
\usepackage{amsfonts}
\usepackage{amssymb}
\usepackage{amsmath}
\usepackage{stmaryrd}
\usepackage{wasysym} 
\usepackage{enumitem}
\usepackage{hyphenat}
\usepackage[colors]{optsys}

\newcommand{\ECapra}{E-\capra}

\title{Capra-Convexity, Convex Factorization and\\ Variational Formulations for the $l_0$ Pseudonorm}

\author{Jean-Philippe Chancelier 
  and Michel De Lara 
  \\ CERMICS, Ecole des Ponts, Marne-la-Vall\'ee, France}

\begin{document}

\maketitle

\begin{abstract}
  The so-called l0~pseudonorm, or cardinality function, counts the number of nonzero components of a vector.
  In this paper, we analyze the l0~pseudonorm
  by means of so-called Capra (constant along primal rays) conjugacies, for which the underlying
  source norm and its dual norm are both orthant-strictly monotonic
  (a notion that we formally introduce and that encompasses the lp norms, but
  for the extreme ones).
  We obtain three main results.
  First, we show that the l0~pseudonorm is equal to its
  Capra-biconjugate, that is, is a Capra-convex function.
  Second, we deduce an unexpected consequence, that we call convex
  factorization: the l0~pseudonorm coincides,  
  on the unit sphere of the source norm, with a proper convex lower semicontinuous
  function. 
  Third, we establish a variational formulation for the l0~pseudonorm
  by means of generalized top-k dual~norms and k-support dual~norms
  (that we formally introduce).
\end{abstract}




\section{Introduction}

The \emph{counting function}, also called \emph{cardinality function}
or \emph{\lzeropseudonorm}, 
counts the number of nonzero components of a vector in~$\RR^d$.
It is used in sparse optimization, either as objective function or in the
constraints, to obtain solutions with few nonzero entries.
However, the mathematical expression of the \lzeropseudonorm\
makes it difficult to be handled as such in optimization problems.
This is why most of the literature on sparse optimization
resorts to \emph{surrogate} problems, 
obtained either from \emph{lower approximations} for the \lzeropseudonorm,
or from \emph{alternative} sparsity-inducing terms
(especially suitable norms).
The literature on sparsity-inducing norms is huge,
and we just point out a very succint part of it.
We refer the reader to \cite{Nikolova:2016} that provides a brief tour of the
literature dealing with least squares minimization constrained by $k$-sparsity,
and to \cite{Hiriart-Urruty-Le:2013} for a survey of the rank
function of a matrix, that shares many properties with the
\lzeropseudonorm.
We refer the reader to \cite{Argyriou-Foygel-Srebro:2012} for the support norm,
to \cite{Tono-Takeda-Gotoh:2017} (and references therein) for top norms,
and to \cite{McDonald-Pontil-Stamos:2016} for generalizations.

Our approach to tackle the \lzeropseudonorm\ uses so-called \capra\
(constant along primal rays) conjugacies,
introduced in~\cite{Chancelier-DeLara:2020_CAPRA_OPTIMIZATION}.
More precisely, in~\cite{Chancelier-DeLara:2020_CAPRA_OPTIMIZATION},
we presented the class of couplings~\capra\
(dependending on an underlying source norm)
and we established expressions for \capra-conjugates and biconjugates,
and \capra-subdifferentials of nondecreasing functions of the \lzeropseudonorm.
In~\cite{Chancelier-DeLara:2021_ECAPRA_JCA},
we introduced the coupling~\ECapra\ related to the Euclidean norm
and we showed that the \lzeropseudonorm\ is \ECapra-convex and 
displays hidden convexity in the following sense.
The \lzeropseudonorm\ satisfies a \emph{convex factorization property}:
it can be written as the composition of a proper convex lower semi continuous (lsc) function
with the normalization mapping that maps any nonzero vector 
onto the Euclidean unit sphere,
hence it coincides with a proper convex lsc function on the Euclidean unit sphere.

In this paper, we go beyond the two above papers in several directions.
We generalize the results
of~\cite{Chancelier-DeLara:2021_ECAPRA_JCA}
by showing that not only the \lzeropseudonorm\
but any nondecreasing function of the \lzeropseudonorm\
is \capra-convex and displays hidden convexity (convex factorization property),
and not only for the Euclidean norm but for a class of norms that encompasses it
(including the $\ell_p$-norms for $p\in ]1,\infty[$).
Moreover, we extend the hidden convexity property to subdifferentials.
Indeed, we show that the \capra-subdifferential
of a nondecreasing function of the \lzeropseudonorm\
coincides, on the unit sphere, with the Rockafellar-Moreau subdifferential of the
associated convex lsc function (in the convex factorization property).
We also add the result that the \capra-subdifferential is a closed convex set.
Whereas, 
in~\cite{Chancelier-DeLara:2020_CAPRA_OPTIMIZATION},
 we obtained \capra-convex lower bounds (inequalities)
for nondecreasing functions of the \lzeropseudonorm,
we now obtain identities.
Whereas we obtained an expression for the \capra-subdifferential
of a nondecreasing function of the \lzeropseudonorm,
we now prove that it is not empty.

The paper is organized as follows.
In Sect.~\ref{CAPRA_convexity_of_the_pseudo_norm_with_SOM_norms},
we provide background on the \lzeropseudonorm\ and on \capra-conjugacies.
We also introduce a new class of orthant-strictly monotonic norms,
as well as sequences of generalized top-$k$ and $k$-support dual~norms.
We show that nondecreasing functions of the \lzeropseudonorm\
are \capra-convex.
In Sect.~\ref{Hidden_convexity_and_variational_formulation_for_the_pseudo_norm},
we show that any nondecreasing function of the \lzeropseudonorm\ coincides, 
when restricted to the unit sphere, with a proper convex lsc function. 
Then, we deduce variational formulations for nondecreasing functions of the \lzeropseudonorm\ 
which involve the sequence of generalized $k$-support dual~norms.
Appendix~\ref{Properties_of_relevant_norms_for_the_lzeropseudonorm} gathers
background on properties of norms that are relevant for the \lzeropseudonorm,
Appendix~\ref{Appendix:Proposition} reproduces
\cite[Proposition~4.5]{Chancelier-DeLara:2020_CAPRA_OPTIMIZATION}
to make easier the reading of proofs, 
and Appendix~\ref{The_Fenchel_conjugacy} gathers background on the Fenchel conjugacy.

\section{\Capra-convexity of the \lzeropseudonorm\ with orthant-strictly monotonic norms}
\label{CAPRA_convexity_of_the_pseudo_norm_with_SOM_norms}

In~\S\ref{The_lzeropseudonorm_and_the_Capra_conjugacy},
we provide background on the \lzeropseudonorm\ and on the family of \capra\
conjugacies (introduced in~\cite{Chancelier-DeLara:2020_CAPRA_OPTIMIZATION}).
Then, in~\S\ref{Relevant_norms_for_the_lzeropseudonorm},
we introduce norms that are especially relevant for the \lzeropseudonorm,
like orthant-strictly monotonic norms.
Finally, in~\S\ref{CAPRA_convexity_of_the_pseudo_norm}, we prove that
the \lzeropseudonorm\ is \capra-convex when the underlying norm
and its dual norm are both orthant-strictly monotonic.

\subsection{Background on the \lzeropseudonorm\ and the \Capra\ conjugacy}
\label{The_lzeropseudonorm_and_the_Capra_conjugacy}

We work on the Euclidean space~$\RR^d$
(where~$d$ is a nonzero integer), equipped with the scalar product 
\( \nscal{\cdot}{\cdot} \) (but not necessarily with the Euclidean norm).
We use the notation \( \ic{j,k}=\na{j, j+1,\ldots,k-1,k} \) for any two 
integers $j$, $k$ such that \( j \leq k \). 

Let $\TripleNorm{\cdot}$ be a norm on~$\RR^d$, 
that we will call the \emph{source norm}.
We denote the unit sphere~$\TripleNormSphere$ and the unit ball~$\TripleNormBall$ 
of the source norm~$\TripleNorm{\cdot}$ by
\begin{equation}
  \TripleNormSphere= 
  \defset{\primal \in \RR^d}{\TripleNorm{\primal} = 1} 
  \eqsepv
  \TripleNormBall = 
  \defset{\primal \in \RR^d}{\TripleNorm{\primal} \leq 1} 
  \eqfinp
  \label{eq:triplenorm_unit_sphere}
\end{equation}

For any vector \( \primal \in \RR^d \), 
\(  \Support{\primal} = \bset{ j \in \ic{1,d} }{\primal_j \not= 0 } 
\subset \ic{1,d} \) is the support of~\( \primal \).
The so-called \emph{\lzeropseudonorm} is the function
\( \lzero : \RR^d \to \ic{0,d} \)
defined by 
\begin{equation}
  \lzero\np{\primal} = \cardinal{ \Support{\primal} }
  = \textrm{number of nonzero components of } \primal
  \eqsepv \forall \primal \in \RR^d
  \eqfinv
  \label{eq:pseudo_norm_l0}  
\end{equation}
where $\cardinal{K}$ denotes the cardinality of 
a subset \( K \subset \ic{1,d} \). 
The \lzeropseudonorm\ shares three out of the four axioms of a norm:
nonnegativity, positivity except for \( \primal =0 \), subadditivity.
The axiom of 1-homogeneity does not hold true. 
By contrast, the \lzeropseudonorm\ is 0-homogeneous:
\begin{equation}
  \lzero\np{\rho\primal} = \lzero\np{\primal} 
  \eqsepv \forall \rho \in \RR\setminus\{0\}
  \eqsepv \forall \primal \in \RR^d
  \eqfinp
  \label{eq:lzeropseudonorm_is_0-homogeneous}
\end{equation}

Following~\cite{Chancelier-DeLara:2020_CAPRA_OPTIMIZATION}, we introduce the coupling~\capra. 

\begin{definition}(\cite[Definition~4.1]{Chancelier-DeLara:2020_CAPRA_OPTIMIZATION})
  The \emph{constant along primal rays coupling}
  $\CouplingCapra: \RR^d\times\RR^d \to \RR$, or \capra,
  between $\RR^d$ and itself, is the function
  \begin{equation}
    \CouplingCapra : \np{\primal,\dual} \in \RR^d\times\RR^d
    \mapsto 
    \begin{cases}
      \frac{ \nscal{\primal}{\dual} }{ \TripleNorm{\primal} }
      \eqsepv
      & \primal \neq 0 \eqfinv
      \\
      0 \eqsepv
      & \text{ else.} 
    \end{cases}
    \label{eq:coupling_CAPRA}
  \end{equation}
\end{definition}
In the case of the Capra coupling, 
    the primal and dual space are the same space~$\RR^d$
    but the Capra coupling is not symmetric in the
    primal and dual variables.
    To stress the point, we use the letter~$\primal$ for a primal vector
    and the letter~$\dual$ for a dual vector. 
We also underline that, in~\eqref{eq:coupling_CAPRA},
the Euclidean scalar product \( \nscal{\primal}{\dual} \)
and the norm term \( \TripleNorm{\primal} \) need not be related, 
that is, the norm~$\TripleNorm{\cdot}$ is not necessarily the Euclidean norm.

The coupling \capra\ has the property of being 
constant along primal rays, hence the acronym~\capra\
(Constant Along Primal RAys).
We introduce 
the primal \emph{normalization mapping}
$\normalized : \RR^d \to  \TripleNormSphere \cup \{0\}$,
from $\RR^d$ towards the unit sphere \( \TripleNormSphere \)
united with $\{0\}$, 
as follows:
\begin{equation}
  \normalized : \primal \in \RR^d
  \mapsto
  \begin{cases}
    \frac{\primal}{\TripleNorm{\primal}}
    & \primal \neq 0 \eqfinv
    \\
    0 \eqsepv
    & \text{ else.} 
  \end{cases}  
  \label{eq:normalization_mapping}
\end{equation}

Now, we introduce notions and notation from generalized convexity
\cite{Singer:1997,Rubinov:2000,Martinez-Legaz:2005}.
As we manipulate functions with values 
in~$\barRR = [-\infty,+\infty] $,
we adopt the Moreau \emph{lower and upper additions} \cite{Moreau:1970} 
that extend the usual addition with 
\( \np{+\infty} \LowPlus \np{-\infty} = \np{-\infty} \LowPlus \np{+\infty} =
-\infty \) or with
\( \np{+\infty} \UppPlus \np{-\infty} = \np{-\infty} \UppPlus \np{+\infty} =
+\infty \).
\begin{subequations}
  The \emph{$\CouplingCapra$-Fenchel-Moreau conjugate} of a 
  function \( \fonctionprimal : \RR^d \to \barRR \), 
  with respect to the coupling~$\CouplingCapra$, is
  the function \( \SFM{\fonctionprimal}{\CouplingCapra} : \RR^d  \to \barRR \) 
  defined by
  \begin{equation}
    \SFM{\fonctionprimal}{\CouplingCapra}\np{\dual} = 
    \sup_{\primal \in \RR^d} \Bp{ \CouplingCapra\np{\primal,\dual} 
      \LowPlus \bp{ -\fonctionprimal\np{\primal} } } 
    \eqsepv \forall \dual \in \RR^d
    \eqfinp
    \label{eq:Fenchel-Moreau_conjugate}
  \end{equation}
  The \emph{$\CouplingCapra$-Fenchel-Moreau biconjugate} of a 
  function \( \fonctionprimal : \RR^d  \to \barRR \), 
  with respect to the coupling~$\CouplingCapra$, is
  the function \( \SFMbi{\fonctionprimal}{\CouplingCapra} : \RR^d \to \barRR \) 
  defined by
  \begin{equation}
    \SFMbi{\fonctionprimal}{\CouplingCapra}\np{\primal} = 
    \sup_{ \dual \in \RR^d } \Bp{ \CouplingCapra\np{\primal,\dual} 
      \LowPlus \bp{ -\SFM{\fonctionprimal}{\CouplingCapra}\np{\dual} } } 
    \eqsepv \forall \primal \in \RR^d 
    \eqfinp
    \label{eq:Fenchel-Moreau_biconjugate}
  \end{equation}
  The biconjugate of a 
  function \( \fonctionprimal : \RR^d  \to \barRR \) satisfies the inequality
  \begin{equation}
    \SFMbi{\fonctionprimal}{\CouplingCapra}\np{\primal}
    \leq \fonctionprimal\np{\primal}
    \eqsepv \forall \primal \in \RR^d 
    \eqfinp
    \label{eq:galois-cor}
  \end{equation}
\end{subequations}
When the coupling~$\CouplingCapra$ is replaced by
the Euclidean scalar product \( \nscal{\cdot}{\cdot} \),
we recover well-known expressions of the Fenchel conjugacy
(see Appendix~\ref{The_Fenchel_conjugacy}).

\subsection{Relevant norms for the \lzeropseudonorm}
\label{Relevant_norms_for_the_lzeropseudonorm}

In~\S\ref{Restriction_norms_and_dual_coordinate-k-norms},
we recall the notions of restriction norms and
of generalized coordinate-$k$ and dual coordinate-$k$ norms.
In~\S\ref{Definitions_of_generalized_top-k_and_k-support_dual_norms},
we introduce two new families of norms, the generalized top-$k$ and $k$-support dual~norms.
Finally, in~\S\ref{Orthant-strictly_monotonic_norms}, we define
a new class of orthant-strictly monotonic norms.

\subsubsection{Restriction norms, generalized coordinate-$k$ norms, dual coordinate-$k$ norms}
\label{Restriction_norms_and_dual_coordinate-k-norms}

For any subset \( K \subset\ic{1,d} \),
we define the subspace
\begin{equation}
  \FlatRR_{K} = 
  \bset{ \primal \in \RR^d }{ \primal_j=0 \eqsepv \forall j \not\in K } 
  \subset \RR^d 
  \label{eq:FlatRR}
\end{equation}
with \( \FlatRR_{\emptyset}=\{0\} \),
and then three norms on the subspace~\( \FlatRR_{K} \) of~\( \RR^d \),
as follows.
\begin{itemize}
\item 
  The \emph{$K$-restriction norm} \( \TripleNorm{\cdot}_{K} \)
  is defined by  \( \TripleNorm{\primal}_{K} = \TripleNorm{\primal} \),
  for any \( \primal \in \FlatRR_{K} \).
\item 
  The $\StarSet{K}$-norm
  \( \TripleNorm{\cdot}_{\star,K} \) is 
  the norm \( \bp{\TripleNorm{\cdot}_{\star}}_{K} \),
  given by the restriction to the subspace~\( \FlatRR_{K} \) of
  the dual norm~$\TripleNormDual{\cdot}$ (first dual, as recalled in
  definition~\eqref{eq:dual_norm} of a dual norm, then restriction),
\item 
  The $\SetStar{K}$-norm
  \( \TripleNorm{\cdot}_{K,\star} \) is 
  the norm \( \bp{\TripleNorm{\cdot}_{K}}_{\star} \),
  given by the dual norm (on the subspace~\( \FlatRR_{K} \))
  of the restriction norm~\( \TripleNorm{\cdot}_{K} \) 
  to the subspace~\( \FlatRR_{K} \) (first restriction, then dual).
\end{itemize}
For any \( \primal \in \RR^d \) and subset \( K \subset \ic{1,d} \), 
we denote by
\( \primal_K \in \FlatRR_{K} \subset \RR^d \) the vector which coincides with~\( \primal \),
except for the components outside of~$K$ that vanish
(this definition is valid for \( K =\emptyset \), giving \( \primal_\emptyset=0 \in
\FlatRR_{\emptyset}=\na{0} \)). 

\begin{definition}(\cite[Definition~3.2]{Chancelier-DeLara:2020_CAPRA_OPTIMIZATION})
  \label{de:coordinate_norm}
  For \( k \in \ic{1,d} \), the expression
  (where the notation \( \sup_{\cardinal{K} \leq k} \) is a shorthand for 
    \( \sup_{ { K \subset \ic{1,d}, \cardinal{K} \leq k}} \)) 
  \begin{equation}
    \CoordinateNormDual{\TripleNorm{\dual}}{k}
    =
    \sup_{\cardinal{K} \leq k} \TripleNorm{\dual_K}_{K,\star} 
    \eqsepv \forall \dual \in \RR^d 
    \label{eq:dual_coordinate_norm_definition}
  \end{equation}
  defines a norm on~$\RR^d$, called the \emph{generalized dual coordinate-$k$ norm} 
  \( \CoordinateNormDual{\TripleNorm{\cdot}}{k} \).
  Its dual norm is the \emph{generalized coordinate-$k$ norm}, denoted by
  \( \CoordinateNorm{\TripleNorm{\cdot}}{k} \).
\end{definition}
We denote the unit sphere~\( \CoordinateNormDual{\TripleNormSphere}{k} \) 
and the unit ball ~\( \CoordinateNormDual{\TripleNormBall}{k} \) by:
\( \forall k\in\ic{1,d} \),
\begin{equation}
  \CoordinateNormDual{\TripleNormSphere}{k} = 
  \defset{\dual \in \RR^d}{\CoordinateNormDual{\TripleNorm{\dual}}{k} = 1} 
  \eqsepv 
  \CoordinateNormDual{\TripleNormBall}{k} = 
  \defset{\dual \in \RR^d}{\CoordinateNormDual{\TripleNorm{\dual}}{k} \leq 1} 
  \eqfinp
  \label{eq:dual_coordinate_norm_unit_sphere_ball}
\end{equation}
We give examples of generalized coordinate-$k$ and dual coordinate-$k$ norms
in \cite[Table~1]{Chancelier-DeLara:2020_CAPRA_OPTIMIZATION}.

\subsubsection{Generalized top-$k$ and $k$-support dual~norms}
\label{Definitions_of_generalized_top-k_and_k-support_dual_norms}

We introduce two new families of norms,
that we call generalized top-$k$ and $k$-support dual~norms.

\begin{definition}  
  \label{de:top_dual_norm}
  For \( k \in \ic{1,d} \), the expression
  \begin{equation}
    \TopDualNorm{\TripleNorm{\dual}}{k}
    =
    \sup_{\cardinal{K} \leq k} \TripleNormDual{\dual_K} 
    =
    \sup_{\cardinal{K} \leq k} \TripleNorm{\dual_K}_{\star,K}
    \eqsepv \forall \dual \in \RR^d 
    \label{eq:top_dual_norm}
  \end{equation}
  defines a norm on~$\RR^d$, called the 
  \emph{generalized top-$k$ dual~norm}
  (associated with the source norm~$\TripleNorm{\cdot}$).
  Its dual norm
  \begin{equation}
    \SupportDualNorm{\TripleNorm{\cdot}}{k} 
    = \bp{ \TopDualNorm{\TripleNorm{\cdot}}{k} }_{\star}
    \eqsepv \forall k\in\ic{1,d}
    \label{eq:support_dual_norm}
  \end{equation}
  is called \emph{generalized $k$-support dual~norm}.
  It has unit sphere~\( \SupportDualNorm{\TripleNormSphere}{k} \)
  and unit ball~\( \SupportDualNorm{\TripleNormBall}{k} \)
  given by: \(  \forall k\in\ic{1,d} \), 
  \begin{equation}
    \SupportDualNorm{\TripleNormSphere}{k}
    =
    \defset{\primal \in \RR^d}{\SupportDualNorm{\TripleNorm{\primal}}{k} = 1} 
    \eqsepv
    \SupportDualNorm{\TripleNormBall}{k}
    =
    \defset{\primal \in \RR^d}{\SupportDualNorm{\TripleNorm{\primal}}{k} \leq 1} 
    \eqfinp
    \label{eq:generalized_k-support_norm_unit}
  \end{equation}
\end{definition}
We use the symbol~$\star$ in the superscript in Equation~\eqref{eq:support_dual_norm}
to indicate that the generalized
    $k$-support dual~norm \( \SupportDualNorm{\TripleNorm{\cdot}}{k} \)
    is a dual norm.
    To stress the point, we use the letter~$\primal$ for a primal vector,
    like in~\( \SupportDualNorm{\TripleNorm{\primal}}{k} \),
    and the letter~$\dual$ for a dual vector,
    like in~\( \TopDualNorm{\TripleNorm{\dual}}{k} \).
    We also adopt the conventions
    \( \TopDualNorm{\TripleNorm{\cdot}}{0} = 0 \)
    and
    \( \SupportDualNorm{\TripleNorm{\cdot}}{0} = 0 \),
    although these are not norms but seminorms.

We now give examples of generalized top-$k$ and $k$-support dual~norms
in the case of $\ell_p$ source norm.
We recall that the $\ell_p$-norms~$\norm{\cdot}_{p}$ on the space~\( \RR^d \) are
defined by \( \norm{\primal}_{p} = \bp{\sum_{i=1}^d |\primal_i|^p}^{\frac{1}{p}}\) 
for $p\in [1,\infty[$, and by
\(\norm{\primal}_{\infty} = \sup_{i\in\ic{1,d}} |\primal_i|\).
It is well-known that the dual norm of the norm~$\norm{\cdot}_{p}$
is the $\ell_q$-norm~$\norm{\cdot}_{q}$, where $q$ is such that \(1/p + 1/q=1\) 
(with the extreme cases $q=\infty$ when $p=1$, and $q=1$ when $p=\infty$). 

We start with a Lemma, whose proof is easy. 
For any \( \dual \in \RR^d \), we denote by \( \module{\dual}
=\np{\module{\dual_1},\ldots,\module{\dual_d}} \)
the vector of~$ \RR^{d}$ with components $|\dual_i|$, $i\in\ic{1,d}$. 
Letting \( \dual\in\RR^d \) and
$\nu$ be a permutation of \( \ic{1,d} \) such that
\( \module{ \dual_{\nu(1)} } \geq \module{ \dual_{\nu(2)} } 
\geq \cdots \geq \module{ \dual_{\nu(d)} } \),
we denote \( \dual^{\downarrow} = \bp{ \module{ \dual_{\nu(1)} }, 
  \module{ \dual_{\nu(2)} }, \ldots, \module{ \dual_{\nu(d)} } } \).
\begin{lemma}
  \label{lem:permutation_invariant_monotonic} 
  Let $\TripleNorm{\cdot}$ be a norm on~$\RR^d$.
  Then, if the norm \( \TripleNorm{\cdot} \) is
  permutation invariant and monotonic 
  --- that is, for any  $\primal$, $\primal'$ in~$ \RR^{d}$, we have 
  \(
  |\primal| \le |\primal'| \Rightarrow 
  \TripleNorm{\primal} \le \TripleNorm{\primal'}
  \),
  where $|\primal| \leq |\primal'|$ means 
  $|\primal_i| \leq |\primal'_i|$ for all $i\in\ic{1,d}$ ---  
  we have that $\TopDualNorm{\TripleNorm{\dual}}{k} = 
  \TripleNormDual{ \dual^{\downarrow}_{ \ic{1,k} }}  $,
  where $\dual^{\downarrow}_{ \ic{1,k} } \in \RR^d$ is given by
  $\np{\dual^{\downarrow}}_{\ic{1,k}}$,
  for all $\dual \in \RR^d$.
\end{lemma}
We first present examples of generalized top-$k$ dual~norms as in~\eqref{eq:top_dual_norm}
(see the second column of Table~\ref{tab:Examples}). 
When the norm~$\TripleNorm{\cdot}$ is the Euclidean norm~$\norm{\cdot}_2$ of~\( \RR^d \),
the generalized top-$k$ dual~norm is known under different names:
the top-$(k,2)$ norm in~\cite[p.~8]{Tono-Takeda-Gotoh:2017},
or the $2$-$k$-symmetric gauge norm \cite{Mirsky:1960}
or the Ky Fan vector norm \cite{Obozinski-Bach:hal-01412385}.
Indeed, in all these cases, the norm of a vector~$\dual$ is obtained
with a subvector of size~$k$ having the~$k$ largest absolute values of the
components, 
because the assumptions of Lemma~\ref{lem:permutation_invariant_monotonic} are satisfied.
More generally, when the norm~$\TripleNorm{\cdot}$ is the $\ell_p$-norm
$\norm{\cdot}_{p}$, for $p\in [1,\infty]$, 
the assumptions of Lemma~\ref{lem:permutation_invariant_monotonic} are also satisfied, as 
$\ell_p$-norms are permutation invariant and monotonic. 
Therefore, we obtain 
that the corresponding generalized top-$k$ dual~norm
\( \TopDualNorm{ \bp{ \norm{\cdot}_{p} } }{k} \) has the expression
\( \TopDualNorm{ \bp{ \norm{\cdot}_{p} } }{k}\np{\dual}
=  \sup_{\cardinal{K} \leq k} {\norm{\dual_K}}_{q} 
= {\norm{ \dual^{\downarrow}_{ \ic{1,k} }}}_{q} \),
for all \( \dual \in \RR^d \), and where \( 1/p+1/q=1 \).
Notice that \( \TopDualNorm{ \bp{ \norm{\cdot}_{p} } }{k} \) is expressed in function
  of~$q$, which can be misleading (this phenomenon is manifest in
  Table~\ref{tab:Examples}). 
When the source norm~$\TripleNorm{\cdot}$ is the $\ell_p$-norm
$\Norm{\cdot}_{p}$, the generalized top-$k$ dual~norm \( \TopDualNorm{\TripleNorm{\cdot}}{k} \)
in~\eqref{eq:top_dual_norm} is called the\emph{\lptopnorm{q}{k}} and is denoted by~\( \LpTopNorm{\cdot}{q}{k} \)
(we invert the indices in the naming convention
  of~\cite[p.~5, p.~8]{Tono-Takeda-Gotoh:2017}, where top-$(k,1)$ and top-$(k,2)$
  were used).
Notice that \( \LpTopNorm{\cdot}{\infty}{k} 
= \norm{\cdot}_{\infty}\) for all~$k\in\ic{1,d}$. 

Now, we turn to examples of generalized $k$-support dual~norms as
in~\eqref{eq:support_dual_norm}
(see the third column of Table~\ref{tab:Examples}).
When the norm~$\TripleNorm{\cdot}$ is the Euclidean norm
\( \norm{\cdot}_2 \) of~\( \RR^d \),
the generalized $k$-support norm is the so-called
$k$-support norm \cite{Argyriou-Foygel-Srebro:2012}. 
More generally, in \cite[Definition 21]{McDonald-Pontil-Stamos:2016},
the authors define the $k$-support $p$-norm or \emph{\lpsupportnorm{p}{k}}
for $p\in [1,\infty]$.
They show, in~\cite[Corollary 22]{McDonald-Pontil-Stamos:2016},
that the dual norm \( \bp{ \TopNorm{ \np{ \norm{\cdot}_{p} } }{k} }_\star \) 
of the above \lptopnorm{p}{k}
is the \lpsupportnorm{q}{k}, where \( 1/p + 1/q = 1 \).
Therefore, the generalized $k$-support dual~norm in~\eqref{eq:support_dual_norm}
is the \lpsupportnorm{p}{k} 
--- denoted by~\( \LpSupportNorm{\cdot}{p}{k} \) ---
when the source norm~$\TripleNorm{\cdot}$ is the $\ell_p$-norm
$\Norm{\cdot}_{p}$, for $p\in [1,\infty]$.
The formula \( \LpSupportNorm{\primal}{\infty}{k} =
\max \na{ \Norm{\primal}_{1} / k , \Norm{\primal}_{\infty} } \)
can be found in~\cite[Exercise IV.1.18, p. 90]{Bhatia:1997}.

\begin{table}
  \centering
  \begin{tabular}{||c||c|c||}
    \hline\hline 
    \small{source norm} \( \TripleNorm{\cdot} \) 
    & \( \TopDualNorm{\TripleNorm{\cdot}}{k} \), $k\in\ic{1,d}$
    & \( \SupportDualNorm{\TripleNorm{\cdot}}{k} \), $k\in\ic{1,d}$
    \\
    \hline\hline 
    \( \Norm{\cdot}_{p} \)
    & \lptopnorm{q}{k} 
    & \lpsupportnorm{p}{k} 
    \\
    & \( \LpTopNorm{\dual}{q}{k} \) %
    & \( \LpSupportNorm{\primal}{p}{k} \) %
    \\
    & \(\LpTopNorm{\dual}{q}{k}=\bp{ \sum_{l=1}^{k} \module{ \dual_{\nu(l)} }^q }^{\frac{1}{q}} \)
    & no analytic expression %
    \\
    \hline
    \( \Norm{\cdot}_{1} \) 
    & \lptopnorm{\infty}{k} 
    & \lpsupportnorm{1}{k} 
    \\
    & $\ell_{\infty}$-norm 
    & $\ell_{1}$-norm
    \\[1mm]
    & &
    \\
    & \( \LpTopNorm{\dual}{\infty}{k} 
      = \Norm{\dual}_{\infty} \), $\forall k\in\ic{1,d}$
    & \( \LpSupportNorm{\primal}{1}{k} = \Norm{\primal}_{1} \), $\forall k\in\ic{1,d}$
    \\
    \hline 
    \( \Norm{\cdot}_{2} \) 
    & \lptopnorm{2}{k} 
    & \lpsupportnorm{2}{k} 
    \\
    & \( \LpTopNorm{\dual}{2}{k} = \sqrt{ \sum_{l=1}^{k} \module{ \dual_{\nu(l)} }^2 } \)
    & \( \LpSupportNorm{\primal}{2}{k} \)
      no analytic expression
    \\
    & & (computation \cite[Prop. 2.1]{Argyriou-Foygel-Srebro:2012})
    \\[1mm]
    & &
    \\
    & \( \LpTopNorm{\dual}{2}{1}= \Norm{\dual}_{\infty} \)
    & \( \LpSupportNorm{\primal}{2}{1} = \Norm{\primal}_{1} \)
      
    \\
    \hline 
    \( \Norm{\cdot}_{\infty} \)
    & \lptopnorm{1}{k} 
    & \lpsupportnorm{\infty}{k} 
    \\
    & \( \LpTopNorm{\dual}{1}{k} = \sum_{l=1}^{k} \module{ \dual_{\nu(l)} } \)
    & \( \LpSupportNorm{\primal}{\infty}{k} =
      \max \na{ \frac{\Norm{\primal}_{1}}{k} , \Norm{\primal}_{\infty} } \) 
    \\[1mm]
    & &
    \\
    & \( \LpTopNorm{\dual}{1}{1}= \Norm{\dual}_{\infty} \)
    & \( \LpSupportNorm{\primal}{1}{1} = \Norm{\primal}_{1} \)

    \\
    \hline\hline
  \end{tabular}
  \caption{Examples of generalized top-$k$ and $k$-support dual~norms
    generated by the $\ell_p$ source norms 
    \( \TripleNorm{\cdot} = \Norm{\cdot}_{p} \) for $p\in [1,\infty]$,
    where $1/p + 1/q =1$.
    For \( \dual \in \RR^d \), $\nu$ denotes a permutation of \( \{1,\ldots,d\} \) such that
    \( \module{ \dual_{\nu(1)} } \geq \module{ \dual_{\nu(2)} } 
    \geq \cdots \geq \module{ \dual_{\nu(d)} } \).
    \label{tab:Examples}}
\end{table}

\subsubsection{Orthant-monotonic and orthant-strictly monotonic norms}
\label{Orthant-strictly_monotonic_norms}

We recall the definition of orthant-monotonic norms
and we introduce the new definition of orthant-strictly monotonic
norms, that will prove especially relevant for the \lzeropseudonorm.

\begin{definition}    
  A norm \( \TripleNorm{\cdot}\) on the space~\( \RR^d \) is called
  \begin{itemize}
  \item 
    \emph{orthant-monotonic} \cite[Definition~2.6]{Gries-Stoer:1967}
    if, for all 
    $\primal$, $\primal'$ in~$ \RR^{d}$, we have 
    \bp{  \(
      |\primal| \le |\primal'| \) and  
      \( \primal~\circ~\primal' \geq~0 \Rightarrow 
      \TripleNorm{\primal} \le \TripleNorm{\primal'}
      \) },
    where $|\primal| \leq |\primal'|$ means 
    $|\primal_i| \leq |\primal^{'}_i|$ for all $i\in\ic{1,d}$, 
    and     where $\primal~\circ~\primal' =
    \np{ \primal_1 \primal'_1,\ldots, \primal_d \primal'_d}$
    is the Hadamard (entrywise) product, 
  \item 
    \emph{orthant-strictly monotonic}
    if, for all 
    $\primal$, $\primal'$ in~$ \RR^{d}$, we have 
    \bp{  \(
      |\primal| < |\primal'| \) and  
      \( \primal~\circ~\primal' \geq~0 \Rightarrow 
      \TripleNorm{\primal} < \TripleNorm{\primal'}
      \) },
    where \( |\primal| < |\primal'| \) means that 
    $|\primal_i| \le |\primal^{'}_i|$ for all $i\in\ic{1,d}$, 
    and there exists $j \in \ic{1,d}$, such that
    $|\primal_j| < |\primal^{'}_j|$.
  \end{itemize}
  \label{de:orthant-monotonic}
\end{definition}
All the $\ell_p$-norms~$\norm{\cdot}_{p}$ on the space~\( \RR^d \), 
for $p\in [1,\infty[$, are strictly monotonic, 
hence orthant-strictly monotonic. By contrast, 
the $\ell_\infty$-norm~$\norm{\cdot}_{\infty}$ is orthant-monotonic but
not orthant-strictly monotonic.

\subsection{\Capra-convexity of the \lzeropseudonorm}
\label{CAPRA_convexity_of_the_pseudo_norm}

The main result of this section is Theorem~\ref{th:pseudonormlzero_conjugate} which 
states that,
when both the source norm~$\TripleNorm{\cdot}$ and its dual norm~$\TripleNormDual{\cdot}$
are orthant-strictly monotonic, then any nondecreasing function of the
\lzeropseudonorm\ is equal to its \capra-biconjugate, that is, is a \capra-convex function.
This considerably generalizes the result in
\cite[Theorem~3.5]{Chancelier-DeLara:2021_ECAPRA_JCA},
which was established for the Euclidean norm and only for the \lzeropseudonorm. 

The proof of Theorem~\ref{th:pseudonormlzero_conjugate} relies on
Proposition~\ref{pr:nonempty_subdifferential}, 
which establishes the nonemptiness of a suitable \capra-subdifferential.
In \cite[Equation~(32)]{Chancelier-DeLara:2020_CAPRA_OPTIMIZATION},
we define the \emph{\capra-subdifferential} of 
the function \( \fonctionprimal : \RR^d \to \barRR \) 
at~\( \primal \in  \RR^d \) by 
\begin{equation}
  \subdifferential{\CouplingCapra}{\fonctionprimal}\np{\primal} 
  =
  \bset{ \dual \in \RR^d }{ %
    \SFM{\fonctionprimal}{\CouplingCapra}\np{\dual}=
    \CouplingCapra\np{\primal, \dual} 
    \LowPlus \bp{ -\fonctionprimal\np{\primal} } }
  \eqfinv
  \label{eq:Capra-subdifferential_b}
\end{equation}
where \( \SFM{\fonctionprimal}{\CouplingCapra}\np{\dual} \)
has been defined in~\eqref{eq:Fenchel-Moreau_conjugate}.

\begin{proposition}
  \label{prop:capra_subdiff}
  Let $\fonctionprimal: \RR^{d} \to \barRR$ be any function.
  For all \( \primal \in \RR^d \),
  the \Capra-subdifferential \(
  \subdifferential{\CouplingCapra}{\fonctionprimal}\np{\primal} \)
  is a closed convex set.
\end{proposition}

\begin{proof}
  We prove that \(
  \subdifferential{\CouplingCapra}{\fonctionprimal}\np{\primal} \),
  as in~\eqref{eq:Capra-subdifferential_b},
  is a closed convex set.
  Let $\primal \in \RR^{d}$.

  By definition~\eqref{eq:Fenchel-Moreau_conjugate}
  of \( \SFM{\fonctionprimal}{\CouplingCapra} \),
  the \Capra-subdifferential~\eqref{eq:Capra-subdifferential_b}
  can be written as
  \[
    \subdifferential{\CouplingCapra}{\fonctionprimal}\np{\primal} 
    =
    \bset{ \dual \in \RR^d }{ %
      \CouplingCapra\np{\primal', \dual} 
      -\fonctionprimal\np{\primal'} 
      \leq 
      \CouplingCapra\np{\primal, \dual} 
      -\fonctionprimal\np{\primal} 
      \eqsepv \forall \primal' \in \RR^{d}}
    \eqsepv 
  \]
  where we use the usual addition
  because \( -\infty < \CouplingCapra\np{\primal,\dual} < +\infty \)
  by~\eqref{eq:coupling_CAPRA}.

  As a consequence, when $\fonctionprimal\np{\primal} = - \infty$, we get that
  $\subdifferential{\CouplingCapra}{\fonctionprimal}\np{\primal} = \RR^{d}$,
  which is closed and convex.
  In the case where $\fonctionprimal\np{\primal} = +
  \infty$, we have that 
  $\subdifferential{\CouplingCapra}{\fonctionprimal}\np{\primal} = \emptyset$ if
  $\fonctionprimal$ is not identically $+\infty$, and that
  $\subdifferential{\CouplingCapra}{\fonctionprimal}\np{\primal} = \RR^{d}$
  otherwise; in either cases, the \Capra-subdifferential is closed and
  convex.
  Now, suppose that $\fonctionprimal\np{\primal} \in \RR$.
  By definition~\eqref{eq:Fenchel-Moreau_conjugate}
  of \( \SFM{\fonctionprimal}{\CouplingCapra} \),
  the \Capra-subdifferential~\eqref{eq:Capra-subdifferential_b}
  can be written as
  \(
  \subdifferential{\CouplingCapra}{\fonctionprimal}\np{\primal} 
  =
  \bset{ \dual \in \RR^d }{ %
    \SFM{\fonctionprimal}{\CouplingCapra}\np{\dual} \leq 
    \CouplingCapra\np{\primal, \dual} 
    -\fonctionprimal\np{\primal} }
  \)
  where the function
  \( \SFM{\fonctionprimal}{\CouplingCapra} \) is a Fenchel conjugate
  by~\cite[Equation~(30b)]{Chancelier-DeLara:2020_CAPRA_OPTIMIZATION},
  hence is closed convex (see the background material in
  Appendix~\ref{The_Fenchel_conjugacy}), 
  and the function~\( g_{\primal} : \RR^d\ni\dual \mapsto \CouplingCapra\np{\primal,
    \dual} -\fonctionprimal\np{\primal}  \) is affine.
  As a consequence,
  \(
  \subdifferential{\CouplingCapra}{\fonctionprimal}\np{\primal} 
  =
  \bset{ \dual \in \RR^d }{ %
    \SFM{\fonctionprimal}{\CouplingCapra}\np{\dual}
    -g_{\primal}\np{\dual} \leq 0 }  \) is a closed convex set.
\end{proof}

It follows that the \capra-subdifferential of the \lzeropseudonorm\ is a
(possibly empty) closed convex set
(by contrast, it is shown in \cite[Section~8]{Hiriart-Urruty-Le:2013} that
  all the generalized [Fenchel] subdifferentials [proximal, Fréchet, viscosity,
  limiting, Clarke] of the rank function coincide and define a vector space).
We now provide conditions under which the \capra-subdifferential of any nondecreasing function of the
\lzeropseudonorm\ is not empty.

\begin{proposition}
  Let $\TripleNorm{\cdot}$ be a norm on~$\RR^d$ with
  associated \capra\ coupling~$\CouplingCapra$ as in~\eqref{eq:coupling_CAPRA}. 
  If both the norm $\TripleNorm{\cdot}$ and the dual norm $\TripleNormDual{\cdot}$
  are orthant-strictly monotonic,
  and if \( \varphi : \ic{0,d} \to \RR \) is a nondecreasing function,
  then
  \begin{equation*}
    \subdifferential{\CouplingCapra}{\np{ \varphi \circ \lzero }}\np{\primal} \neq \emptyset
    \eqsepv \forall \primal \in \RR^d 
    \eqfinp 
  \end{equation*}
  More precisely, when $\primal=0$, we have that
  \( \subdifferential{\CouplingCapra}{\np{ \varphi \circ \lzero }}\np{0}
  = \bigcap_{ \LocalIndex\in\ic{1,d} } \bc{ \varphi\np{\LocalIndex} -\varphi\np{0} }
  \CoordinateNormDual{\TripleNormBall}{\LocalIndex} \)
  \( \neq~\emptyset \),
  where the unit ball~\( \CoordinateNormDual{\TripleNormBall}{k} \) is defined
  in~\eqref{eq:dual_coordinate_norm_unit_sphere_ball}.
  When $\primal\not=0$, there exists \( \dual \in \RR^d \) satisfying 
  \( \Support{\dual} = \Support{\primal} \) 
  and \( \nscal{\primal}{\dual} =
  \TripleNorm{\primal} \times \TripleNormDual{\dual} \), and for all such \( \dual \in \RR^d \)
  we have that \( \lambda \dual \in \subdifferential{\CouplingCapra}{\np{ \varphi \circ \lzero }}\np{\primal} \)
  for $\lambda >0$ large enough.
  \label{pr:nonempty_subdifferential}
\end{proposition}

\begin{proof}
  The proof relies on results established in
  Appendix~\ref{Properties_of_relevant_norms_for_the_lzeropseudonorm}.

  Since the norm $\TripleNorm{\cdot}$ is orthant-strictly monotonic,
  it is orthant-monotonic, so that we have 
  \( \CoordinateNorm{\TripleNorm{\cdot}}{\LocalIndex} 
  =
  \SupportDualNorm{\TripleNorm{\cdot}}{\LocalIndex} \) 
  and
  \(      \CoordinateNormDual{\TripleNorm{\cdot}}{\LocalIndex}
  =
  \TopDualNorm{\TripleNorm{\cdot}}{\LocalIndex} \),
  for \( \LocalIndex\in\ic{0,d} \) 
  by~\eqref{eq:dual_coordinate-k_norm_=_generalized_top-k_norm} in
  Proposition~\ref{pr:dual_coordinate-k_norm_=_generalized_top-k_norm} 
  in Appendix~\ref{Coordinate-k_and_dual_coordinate-k_norms}
  (with the convention that these are the null seminorms in the case \( \LocalIndex=0 \)).
  Therefore, we can translate all the results with 
  generalized top-$k$ and $k$-support dual~norms
  (Definition~\ref{de:top_dual_norm}) instead of 
  coordinate-$k$ and dual coordinate-$k$ norms (Definition~\ref{de:coordinate_norm}).
  \medskip

  When \( \primal=0 \), we have by
  \cite[Equation~(39) in Proposition~4.7]{Chancelier-DeLara:2020_CAPRA_OPTIMIZATION}
  that
  \begin{equation*}
    \subdifferential{\CouplingCapra}{\np{ \varphi \circ \lzero}}\np{0}
    = \bigcap_{ \LocalIndex\in\ic{1,d} } \bc{ \varphi\np{\LocalIndex} \UppPlus \bp{-\varphi\np{0} } } 
    \CoordinateNormDual{\TripleNormBall}{\LocalIndex} 
    = \bigcap_{ \LocalIndex\in\ic{1,d} } \bc{ \varphi\np{\LocalIndex} -\varphi\np{0} }
    \CoordinateNormDual{\TripleNormBall}{\LocalIndex}    \ni 0 
    \eqfinv
  \end{equation*}
  because \( \varphi\np{\LocalIndex} \UppPlus \bp{-\varphi\np{0}}
  = \varphi\np{\LocalIndex} -\varphi\np{0} \geq 0 \)
  since \( \varphi : \ic{0,d} \to \RR \) is a nondecreasing function.
  \medskip

  From now on, we consider $\primal \in \RR^d\setminus\{0\} $ such that 
  $\lzero\np{\primal}=l \in \ic{1,d}$,
  and we will use the following equivalence, established
  in~\cite[Equation~(40) in Proposition~4.7]{Chancelier-DeLara:2020_CAPRA_OPTIMIZATION}
  \begin{equation}
    \dual \in \subdifferential{\CouplingCapra}{\np{ \varphi \circ \lzero}}\np{\primal} 
    \iff 
    \begin{cases}
      \dual \in 
      \NORMAL_{ \CoordinateNorm{\TripleNormBall}{l} }
      \np{\frac{ \primal }{ \CoordinateNorm{ \TripleNorm{\primal} }{l} } }
      \\
      \mtext{and } 
      l \in \argmax_{\LocalIndex\in\ic{0,d}} \bc{ \CoordinateNormDual{\TripleNorm{\dual}}{\LocalIndex}-\varphi\np{\LocalIndex} }
      \eqfinv
    \end{cases}
    \label{eq:pseudonormlzero_subdifferential}
  \end{equation}
  where the normal cone \( \NORMAL_{ \CoordinateNorm{\TripleNormBall}{l} } 
  \np{\frac{ \primal }{ \CoordinateNorm{ \TripleNorm{\primal} }{l} } } \)
  is defined in~\eqref{eq:normal_cone}.

  Since the norm $\TripleNorm{\cdot}$ is orthant-strictly monotonic,
  we know by Item~\ref{it:SDC} in
  Proposition~\ref{pr:orthant-strictly_monotonic}
  (Appendix~\ref{Properties_of_Orthant-strictly_monotonic_norms})
  that there exists a vector \( \dual \in \RR^d \) such that 
  \begin{subequations}
    \begin{align}
      L = \Support{\primal} = \Support{\dual} \mtext{ hence }   \lzero\np{\dual}
      &=
        \lzero\np{\primal}=l>1
        \eqfinv 
        \label{eq:conjugate_point_proof_a}
      \\
      \nscal{\primal}{\dual}
      &= 
        \TripleNorm{\primal} \times \TripleNormDual{\dual} 
        \eqfinp
        \label{eq:conjugate_point_proof_b}
    \end{align}
    \label{eq:conjugate_point_proof}
  \end{subequations}
  Since the dual norm $\TripleNormDual{\cdot}$ is orthant-strictly monotonic,
  we know by~\eqref{eq:level_curve_l0_characterization}
  in Proposition~\ref{pr:assumptions_pseudonormlzero_conjugate}
  (Appendix~\ref{Coordinate-k_and_dual_coordinate-k_norms}) that
  \begin{equation}
    \CoordinateNormDual{\TripleNorm{\dual}}{1} < \cdots < 
    \CoordinateNormDual{\TripleNorm{\dual}}{l-1}  < 
    \CoordinateNormDual{\TripleNorm{\dual}}{l} =
    \cdots = \CoordinateNormDual{\TripleNorm{\dual}}{d} =\TripleNormDual{\dual}
    \eqfinp
    \label{eq:level_curve_l0_characterization_bis}
  \end{equation}
  We now show that \( \dual \in \subdifferential{\CouplingCapra}{\np{ \varphi \circ \lzero }}\np{\primal} \). 
  \medskip

  First, we are going to establish that 
  \( \dual \in \NORMAL_{ \CoordinateNorm{\TripleNormBall}{l} }
  \np{\frac{ \primal }{ \CoordinateNorm{ \TripleNorm{\primal} }{l} } } \), that
  is, the first of the two conditions in the characterization~\eqref{eq:pseudonormlzero_subdifferential}
  of the subdifferential~\( \subdifferential{\CouplingCapra}{\np{ \varphi \circ \lzero }}\np{\primal} \). 

  On the one hand, because \( \lzero\np{\dual}=l \) and
  by~\eqref{eq:level_curve_l0_characterization_bis},
  we have that \( \TripleNormDual{\dual} =
  \CoordinateNormDual{\TripleNorm{\dual}}{l} \). 
  On the other hand, because \( \lzero\np{\primal}=l \) we have that
  \( \TripleNorm{\primal} = \CoordinateNorm{\TripleNorm{\primal}}{l} \)
  by \cite[Equation~(25a)]{Chancelier-DeLara:2020_CAPRA_OPTIMIZATION}. 
  Hence, from~\eqref{eq:conjugate_point_proof_b},
  we get that \( \nscal{\primal}{\dual} =
  \CoordinateNorm{\TripleNorm{\primal}}{l} 
  \times \CoordinateNormDual{\TripleNorm{\dual}}{l} \),
  from which we obtain 
  \( \dual \in \NORMAL_{ \CoordinateNorm{\TripleNormBall}{l} }
  \np{\frac{ \primal }{ \CoordinateNorm{ \TripleNorm{\primal} }{l} } } \)
  by property~\eqref{eq:couple_TripleNorm-dual_and_normal_cone} of the normal cone
  as \( \primal \neq 0 \).
  To close this part, notice that, for all $\lambda > 0 $,
  we have that 
  \(  \lambda\dual \in 
  \NORMAL_{ \CoordinateNorm{\TripleNormBall}{l} }
  \np{\frac{ \primal }{ \CoordinateNorm{ \TripleNorm{\primal} }{l} } } \), 
  because this last set is a cone. 
  \medskip

  Second, we prove the other of the two conditions
  in the characterization~\eqref{eq:pseudonormlzero_subdifferential}
  of the subdifferential~\( \subdifferential{\CouplingCapra}{\np{ \varphi \circ \lzero }}\np{\primal} \).
  More precisely, we are going to show that, for $\lambda$ large enough,
  \( \CoordinateNormDual{\TripleNorm{\lambda\dual}}{l}-\varphi\np{l} = 
  \sup_{\LocalIndex\in\ic{0,d}} \bc{ \CoordinateNormDual{\TripleNorm{\lambda\dual}}{\LocalIndex}-\varphi\np{\LocalIndex} }
  \).
  For this purpose, we consider the mapping $\psi:\; ]0,+\infty[ \to \RR$ defined by
  \begin{equation*}
    \psi(\lambda) =
    \CoordinateNormDual{\TripleNorm{\lambda\dual}}{l}-\varphi\np{l} -
    \sup_{\LocalIndex\in\ic{0,d}} \bc{ \CoordinateNormDual{\TripleNorm{\lambda\dual}}{\LocalIndex}-\varphi\np{\LocalIndex} }
    \eqsepv \forall \lambda > 0 
    \eqfinv
  \end{equation*}
  and we are going to show that \( \psi(\lambda) =0 \) for $\lambda$ large enough.
  We have 
  \begin{align*}
    \psi(\lambda) 
    &=
      \inf_{\LocalIndex\in\ic{0,d}} \Bp{ \lambda \bp{ \CoordinateNormDual{\TripleNorm{\dual}}{l} -
      \CoordinateNormDual{\TripleNorm{\dual}}{\LocalIndex} } + \varphi\np{\LocalIndex}-\varphi\np{l} }
    \\
    &=
      \inf \left\{ \lambda \CoordinateNormDual{\TripleNorm{\dual}}{l}+\varphi\np{0}-\varphi\np{l} ,
      \inf_{\LocalIndex\in\ic{1,l-1}} \Bp{ \lambda \bp{ \CoordinateNormDual{\TripleNorm{\dual}}{l} -
      \CoordinateNormDual{\TripleNorm{\dual}}{\LocalIndex} } +
      \varphi\np{\LocalIndex}-\varphi\np{l} } , \right.
      \tag{as \( \CoordinateNormDual{\TripleNorm{\dual}}{0}=0 \) by convention}
    \\ 
    & \phantom{\inf\qquad}
      \left.
      \inf_{\LocalIndex\in\ic{l,d}} \Bp{ \lambda \bp{ \CoordinateNormDual{\TripleNorm{\dual}}{l} -
      \CoordinateNormDual{\TripleNorm{\dual}}{\LocalIndex} } + \varphi\np{\LocalIndex}-\varphi\np{l} } \right\}
    \\
    &=
      \inf \left\{ \lambda \CoordinateNormDual{\TripleNorm{\dual}}{l}+\varphi\np{0}-\varphi\np{l} ,
      \inf_{\LocalIndex\in\ic{1,l-1}} \Bp{ \lambda \bp{ \CoordinateNormDual{\TripleNorm{\dual}}{l} -
      \CoordinateNormDual{\TripleNorm{\dual}}{\LocalIndex} } + \varphi\np{\LocalIndex}-\varphi\np{l} } , \right.
    \\ 
    & \phantom{\inf\qquad}
      \left. \inf_{\LocalIndex\in\ic{l,d}} \bp{
      \varphi\np{\LocalIndex}-\varphi\np{l} } \right\}
      \tag{as \( \CoordinateNormDual{\TripleNorm{\dual}}{\LocalIndex}
      = \CoordinateNormDual{\TripleNorm{\dual}}{l} \) for
      $\LocalIndex \geq l$ by~\eqref{eq:level_curve_l0_characterization_bis}}
    \\
    &=
      \inf \Ba{ \lambda \CoordinateNormDual{\TripleNorm{\dual}}{l}+\varphi\np{0}-\varphi\np{l} ,
      \inf_{\LocalIndex\in\ic{1,l-1}} \Bp{ \lambda \bp{ \CoordinateNormDual{\TripleNorm{\dual}}{l} -
      \CoordinateNormDual{\TripleNorm{\dual}}{\LocalIndex} } +
      \varphi\np{\LocalIndex}-\varphi\np{l} } , 0 } 
      \eqfinv
  \end{align*}
  as \(\inf_{\LocalIndex\in\ic{l,d}} \bp{
    \varphi\np{\LocalIndex}-\varphi\np{l} }=0 \)
  because \( \varphi : \ic{0,d} \to \RR \) is a nondecreasing function.
  Let us show that the two first terms in the infimum
  go to $+\infty$ when \( \lambda \to +\infty \).
  The first term \( \lambda \CoordinateNormDual{\TripleNorm{\dual}}{l}+\varphi\np{0}-\varphi\np{l} \) goes to $+\infty$ 
  because, by~\eqref{eq:level_curve_l0_characterization_bis}, 
  we have that \( \CoordinateNormDual{\TripleNorm{\dual}}{l}=\TripleNormDual{\dual}>0 \) 
  as \( \dual \in \RR^d\setminus\{0\} \) since $\lzero\np{\dual}=l \geq 1$. 
  The second term \( \inf_{\LocalIndex\in\ic{1,l-1}} \Big( \lambda \bp{ \CoordinateNormDual{\TripleNorm{\dual}}{l} -
    \CoordinateNormDual{\TripleNorm{\dual}}{\LocalIndex} } \)
  \( +~\varphi\np{\LocalIndex}-\varphi\np{l} \Big) \)
  also goes to $+\infty$ because 
  $\lzero\np{\dual}=l  \geq 1$, so that 
  \( \TripleNormDual{\dual}=\CoordinateNormDual{\TripleNorm{\dual}}{l} > 
  \CoordinateNormDual{\TripleNorm{\dual}}{\LocalIndex} \)
  for $\LocalIndex\in\ic{1,l-1}$
  again by~\eqref{eq:level_curve_l0_characterization_bis}.
  Therefore, we deduce that
  \( \psi(\lambda) =0 \) for $\lambda$ large enough,
  and thus 
  \( \CoordinateNormDual{\TripleNorm{\lambda\dual}}{l}-\varphi\np{l} =
  \sup_{\LocalIndex\in\ic{0,d}} \bc{ \CoordinateNormDual{\TripleNorm{\lambda\dual}}{\LocalIndex}-\varphi\np{\LocalIndex} }
  \), that is, 
  \( l \in \argmax_{\LocalIndex\in\ic{0,d}} \bc{
    \CoordinateNormDual{\TripleNorm{\lambda\dual}}{\LocalIndex}-\varphi\np{\LocalIndex} } \).
  \medskip

  Wrapping up the above results, we have shown that, 
  for any vector \( \dual \in \RR^d \) such that 
  \( \Support{\dual} = \Support{\primal} \),
  and that \( \nscal{\primal}{\dual} =
  \TripleNorm{\primal} \times \TripleNormDual{\dual} \),
  then, for $\lambda >0$ large enough,
  \( \lambda\dual \) satisfies the two conditions 
  in the characterization~\eqref{eq:pseudonormlzero_subdifferential}
  of the subdifferential~\( 
  \subdifferential{\CouplingCapra}{\np{ \varphi \circ \lzero }}\np{\primal} \).
  Hence, we get that \( \lambda\dual \in \subdifferential{\CouplingCapra}{\np{ \varphi \circ \lzero
    }}\np{\primal} \).
  \medskip

  This ends the proof. 
\end{proof}

\begin{theorem}
  \label{th:pseudonormlzero_conjugate}
  Let $\TripleNorm{\cdot}$ be a norm on~$\RR^d$ with associated sequence
  \( \bseqa{\TopDualNorm{\TripleNorm{\cdot}}{\LocalIndex}}{\LocalIndex\in\ic{1,d}} \)
  of generalized top-$k$ dual~norms,
  as in Definition~\ref{de:top_dual_norm}, 
  and with
  associated \capra\ coupling~$\CouplingCapra$ as in~\eqref{eq:coupling_CAPRA}. 

  \noindent  If the norm $\TripleNorm{\cdot}$ is orthant-monotonic, then, for any 
  function \( \varphi : \ic{0,d} \to \RR \), we have that
  (we recall the convention
    that \( \TopDualNorm{\TripleNorm{\cdot}}{0} = 0 \))
  \begin{align}
    \SFM{ \np{ \varphi \circ \lzero } }{\CouplingCapra} 
    &=
      \sup_{\LocalIndex\in\ic{0,d}} \Bc{ \TopDualNorm{\TripleNorm{\cdot}}{\LocalIndex}-\varphi\np{\LocalIndex} }  
      \eqfinp
      \label{eq:conjugate_l0norm_TopDualNorm}
      \intertext{If both the norm $\TripleNorm{\cdot}$ and the dual norm $\TripleNormDual{\cdot}$
      are orthant-strictly monotonic, then, for any 
      nondecreasing function \( \varphi : \ic{0,d} \to \RR \), we have that} 
      \SFMbi{ \np{ \varphi \circ \lzero } }{\CouplingCapra} 
    &= \varphi \circ \lzero 
      \eqfinp
      \label{eq:biconjugate_l0norm_TopDualNorm}
  \end{align}
\end{theorem}

\begin{proof}
  Suppose that the norm $\TripleNorm{\cdot}$ is orthant-monotonic.
  Then, by~\eqref{eq:dual_coordinate-k_norm_=_generalized_top-k_norm} in
  Proposition~\ref{pr:dual_coordinate-k_norm_=_generalized_top-k_norm}
  (Appendix~\ref{Coordinate-k_and_dual_coordinate-k_norms}),
  we get that   
  \( \CoordinateNorm{\TripleNorm{\cdot}}{k} 
  =
  \SupportDualNorm{\TripleNorm{\cdot}}{k} \)
  and
  \( \CoordinateNormDual{\TripleNorm{\cdot}}{k}
  =
  \TopDualNorm{\TripleNorm{\cdot}}{k} \).
  Moreover, it is proved 
  in~\cite[Equation~(33) in Proposition~4.4]{Chancelier-DeLara:2020_CAPRA_OPTIMIZATION}
  that \( \SFM{ \np{ \varphi \circ \lzero } }{\CouplingCapra} = \) \(
  \sup_{\LocalIndex\in\ic{0,d}} \Bc{ \CoordinateNormDual{\TripleNorm{\cdot}}{\LocalIndex}-\varphi\np{\LocalIndex} } \).
  As \( \CoordinateNormDual{\TripleNorm{\cdot}}{\LocalIndex}
  =
  \TopDualNorm{\TripleNorm{\cdot}}{\LocalIndex} \), we obtain~\eqref{eq:conjugate_l0norm_TopDualNorm}. 
  \medskip

  Suppose that both the norm $\TripleNorm{\cdot}$ and the dual norm $\TripleNormDual{\cdot}$
  are orthant-strictly monotonic. Then, 
  Proposition~\ref{pr:nonempty_subdifferential} applies.
  Therefore, for any vector~\( \primal \in \RR^d \)
  and any \( \dual \in \subdifferential{\CouplingCapra}{\np{ \varphi \circ \lzero }}\np{\primal}  \neq \emptyset \),
  we obtain
  \begin{align*}
    \SFMbi{\np{ \varphi \circ \lzero }}{\CouplingCapra}\np{\primal} 
    & \geq 
      \CouplingCapra\np{\primal,\dual} \LowPlus
      \bp{ - \SFM{\np{ \varphi \circ \lzero }}{\CouplingCapra}\np{\dual} }
      \tag{by definition~\eqref{eq:Fenchel-Moreau_biconjugate} of the biconjugate}
    \\
    &=
      \CouplingCapra\np{\primal,\dual} 
      - \SFM{\np{ \varphi \circ \lzero }}{\CouplingCapra}\np{\dual}
      \tag{because \( -\infty < \CouplingCapra\np{\primal,\dual} < +\infty \)
      by~\eqref{eq:coupling_CAPRA}, so that the usual addition applies}
    \\
    &=
      \CouplingCapra\np{\primal,\dual} 
      - \bp{ \CouplingCapra\np{\primal,\dual} - \np{ \varphi \circ \lzero }\np{\primal} }
      \intertext{by definition~\eqref{eq:Capra-subdifferential_b} 
      of the \capra-subdifferential \( \subdifferential{\CouplingCapra}{\np{
      \varphi \circ \lzero }}\np{\primal} \),
      and where again we can use the usual addition}
    &=
      \np{ \varphi \circ \lzero }\np{\primal} 
      \eqfinp
  \end{align*}
  On the other hand, we have that 
  \( \SFMbi{\np{ \varphi \circ \lzero }}{\CouplingCapra}\np{\primal} \leq \np{ \varphi \circ \lzero }\np{\primal} \) 
  by~\eqref{eq:galois-cor}.
  We conclude that \( \SFMbi{\np{ \varphi \circ \lzero }}{\CouplingCapra}\np{\primal} =
  \np{ \varphi \circ \lzero }\np{\primal} \), which is~\eqref{eq:biconjugate_l0norm_TopDualNorm}.  
  \medskip

  This ends the proof.
\end{proof}

Our proof of Proposition~\ref{pr:nonempty_subdifferential}, hence of
Theorem~\ref{th:pseudonormlzero_conjugate}, uses the property that the nondecreasing function
\( \varphi : \ic{0,d} \to \RR \) takes finite values.
What happens for a function~\( \varphi \) with values in the extended real numbers?
A typical case of \( \varphi : \ic{0,d} \to \barRR \) is that of the characteristic function of a
subset \( K \subset \ic{1,d} \): \( \delta_{K}\np{\LocalIndex} = 0 \) if \( \LocalIndex \in K \),
and \( \delta_{K}\np{\LocalIndex} = +\infty \) if \( \LocalIndex \not\in K \).
Regarding \capra-convexity of a nondecreasing function of the
\lzeropseudonorm\ taking infinite values, we suspect that a proof would rely on
different assumptions than those of
Proposition~\ref{pr:nonempty_subdifferential}
and Theorem~\ref{th:pseudonormlzero_conjugate}.
%
As an indication, the comment after the proof of 
\cite[Corollary~4.6]{Chancelier-DeLara:2020_CAPRA_OPTIMIZATION}
points out that, when the normed space 
\( \bp{\RR^d,\TripleNorm{\cdot}} \) is strictly convex
(that is, when the unit ball~$\TripleNormBall$ is rotund),
then \( \delta_{\ic{0,k}} \circ \lzero \) is $\CouplingCapra$-convex
for \( k\in \ic{0,d} \).
As the normed space \( \bp{\RR^d,\norm{\cdot}_{p}} \), equipped with 
the $\ell_p$-norm~$\norm{\cdot}_{p}$ (for $p\in [1,\infty]$), is strictly
convex if and only if $p\in ]1,\infty[$,
we get that the characteristic functions of the level sets of the
\lzeropseudonorm\ are $\CouplingCapra$-convex when the source norm
\( \TripleNorm{\cdot}=\norm{\cdot}_{p} \) for $p\in ]1,\infty[$.


\section{Convex factorization and variational formulation for the \lzeropseudonorm}
\label{Hidden_convexity_and_variational_formulation_for_the_pseudo_norm}

In this section, we suppose that both the norm~$\TripleNorm{\cdot}$ and the dual norm~$\TripleNormDual{\cdot}$
are orthant-strictly monotonic.
In~\S\ref{Hidden_convexity_in_the_lzeropseudonorm}, 
we show that any nonnegative nondecreasing function of the pseudonorm~$\lzero$ coincides, 
on the unit sphere, with a proper convex lsc function on~$\RR^d$,
and we provide various expressions for this latter function.  
In~\S\ref{Variational_formulation_for_the_pseudo_norm}, 
we deduce a variational formula for nonnegative nondecreasing functions of the
\lzeropseudonorm.

\subsection{Convex factorization and hidden convexity in the \lzeropseudonorm}
\label{Hidden_convexity_in_the_lzeropseudonorm}

We now present a (rather unexpected) consequence of the just established
property that any nondecreasing function of the pseudonorm~$\lzero$ is \capra-convex
(Theorem~\ref{th:pseudonormlzero_conjugate}).
Indeed, we prove that any nonnegative nondecreasing function of the pseudonorm~$\lzero$ coincides, 
on the unit sphere~$\TripleNormSphere =
\bset{\primal \in \RR^d}{\TripleNorm{\primal} = 1} $, 
with a proper convex lsc function on~$\RR^d$,
and that this property extends to subdifferentials.
We also provide various expressions for the underlying proper convex lsc function.

In this part, we make use of the Fenchel conjugacy,
denoted by~\( \LFM{} \) in~\eqref{eq:Fenchel_conjugate}, and of the reverse
Fenchel conjugacy~\( \LFMr{} \) in~\eqref{eq:Fenchel_conjugate_reverse}
(see Appendix~\ref{The_Fenchel_conjugacy}).
  In the case of the Fenchel conjugacy, 
  the primal and dual space are the same space~$\RR^d$
  \emph{and} the scalar product coupling is \emph{symmetric} in the
  primal and dual variables. Hence, we could use the notation~\( \LFM{} \)
  in~\eqref{eq:Fenchel_conjugate_reverse} instead of~\( \LFMr{} \).
  By contrast, in the case of the Capra conjugacy, 
  the primal and dual space are the same space~$\RR^d$
  \emph{but} the Capra coupling is \emph{not symmetric} in the
  primal and dual variables.
  Hence, $\CouplingCapra' \neq \CouplingCapra$ and we cannot use the notation ${}^{\CouplingCapra\CouplingCapra}$,
  but we have to use the notation~${}^{\CouplingCapra\CouplingCapra'}$.
  For the sake of consistency, we maintain the notations~\( \LFMr{} \) and
  ~\( \LFMbi{} \).

\begin{proposition}
  Let $\TripleNorm{\cdot}$ be a norm on~$\RR^d$ with associated sequence
  \( \bseqa{\TopDualNorm{\TripleNorm{\cdot}}{\LocalIndex}}{\LocalIndex\in\ic{1,d}} \)
  of generalized top-$k$ dual~norms, 
  and sequence \(
  \bseqa{\SupportDualNorm{\TripleNorm{\cdot}}{\LocalIndex}}{\LocalIndex\in\ic{1,d}}
  \) of generalized $k$-support dual~norms,
  and with
  associated \capra\ coupling~$\CouplingCapra$. 
  Suppose that both the norm $\TripleNorm{\cdot}$ and the dual norm $\TripleNormDual{\cdot}$
  are orthant-strictly monotonic.
  Let \( \varphi : \ic{0,d} \to \RR_+ \) be a nonnegative nondecreasing function,
  such that \( \varphi\np{0}=0 \).
  We define the function~${\cal L}_{0}^{\varphi} : \RR^d \to \barRR $  by 
  \begin{equation}
    {\cal L}_{0}^{\varphi} = \LFMr{ \bp{ \SFM{\np{ \varphi \circ \lzero }}{\CouplingCapra} } }
    \eqfinp
    \label{eq:definition_calL_0}  
  \end{equation}
  Then, the following statements hold true.
  \begin{enumerate}[label=({\bf\emph{\alph*}})]
    \begin{subequations}
    \item  \label{pr:one}
      The function~${\cal L}_{0}^{\varphi} : \RR^d \to \barRR $  
      is proper convex lsc.
    \item  \label{pr:two}
      The function~$\varphi \circ \lzero$ coincides, on the unit sphere~$\TripleNormSphere =
      \bset{\primal \in \RR^d}{\TripleNorm{\primal} = 1} $, with the
      function~${\cal L}_{0}^{\varphi}$, that is, 
      \begin{equation}
        \np{ \varphi \circ \lzero }\np{\sphere} = {\cal L}_{0}^{\varphi}\np{\sphere} 
        \eqsepv 
        \forall \sphere \in \TripleNormSphere 
        \eqfinp
        \label{eq:pseudonormlzero_convex_sphere}
      \end{equation}
    \item \label{pr:three}
      The \capra-subdifferential 
      of the function~$\varphi \circ \lzero$
      coincides, on the unit sphere~$\TripleNormSphere$,
      with the (Rockafellar-Moreau) subdifferential 
      of the function~${\cal L}_{0}^{\varphi}$, that is, 
      \begin{equation}
        \subdifferential{\CouplingCapra}{\np{ \varphi \circ \lzero }}\np{\sphere}
        =  \subdifferential{}{{\cal L}_{0}^{\varphi}}\np{\sphere} 
        \eqsepv 
        \forall \sphere \in \TripleNormSphere 
        \eqfinp
        \label{eq:Capra-subdifferential=Rockafellar-Moreau-subdifferential}
      \end{equation}
    \item \label{pr:four}
      Convex factorization property.       
      The function~$\varphi \circ \lzero$ can be expressed as the 
      composition of the proper convex lsc function~${\cal L}_{0}^{\varphi}$
      with the 
      normalization mapping~$\normalized$
      , that is,
      \begin{align}
        \varphi \circ \lzero
        &=
          {\cal L}_{0}^{\varphi} \circ \normalized
          \label{eq:pseudonormlzero_convex_normalization_mapping}
          \intertext{ or, equivalently, }
          \np{ \varphi \circ \lzero }\np{\primal}
        &= 
          {\cal L}_{0}^{\varphi} \Bp{ \frac{ \primal }{ \TripleNorm{\primal} } } 
          \eqsepv 
          \forall \primal \in \RR^d\setminus\{0\}
          \eqfinp 
          \label{eq:pseudonormlzero_convex_normalization_mapping_bis}
      \end{align}
    \end{subequations}
  \item \label{pr:five}
    \begin{subequations}
      The function~${\cal L}_{0}^{\varphi}$ is given by
      \begin{equation}
        {\cal L}_{0}^{\varphi}= 
        \LFMr{ \bgp{ \sup_{\LocalIndex\in\ic{0,d}} 
            \Bc{ \TopDualNorm{\TripleNorm{\cdot}}{\LocalIndex}-\varphi\np{\LocalIndex} } } }
        \eqfinp
        \label{eq:calL_0_definition}
      \end{equation}
    \item \label{pr:six}
      The function~${\cal L}_{0}^{\varphi}$ is the largest convex lsc function
      below the integer valued function 
      \begin{equation}
        \RR^d \ni \primal \mapsto 
        \inf_{\LocalIndex\in\ic{0,d}} \Bc{ \delta_{
            \SupportDualNorm{\TripleNormBall}{\LocalIndex} }\np{\primal} + \varphi\np{\LocalIndex} }
        \eqsepv
        \label{eq:pseudonormlzero_largest_convex_balls}
      \end{equation}
      that is, below the function
      $\primal \in \SupportDualNorm{\TripleNormBall}{\LocalIndex} \setminus \SupportDualNorm{\TripleNormBall}{\LocalIndex-1}
      \mapsto \varphi\np{\LocalIndex}$ for $\LocalIndex\in\ic{1,d}$ and  $\primal \in
      \SupportDualNorm{\TripleNormBall}{0} = \{0\} \mapsto 0$,
      the function
      being infinite outside~\( \SupportDualNorm{\TripleNormBall}{d} = \BALL \)
      (the above construction makes sense as \(
      \SupportDualNorm{\TripleNormBall}{1}
      \subset \cdots \subset
      \SupportDualNorm{\TripleNormBall}{\LocalIndex-1} \subset \SupportDualNorm{\TripleNormBall}{\LocalIndex} 
      \subset \cdots \subset \SupportDualNorm{\TripleNormBall}{d} = \BALL \)). 
    \item \label{pr:seven}
      The function~${\cal L}_{0}^{\varphi}$ is the largest convex lsc function
      below the integer valued function 
      \begin{equation}
        \RR^d \ni \primal \mapsto 
        \inf_{\LocalIndex\in\ic{0,d}} \Bc{ \delta_{
            \SupportDualNorm{\TripleNormSphere}{\LocalIndex} }\np{\primal} +
          \varphi\np{\LocalIndex} }
        \eqsepv
        \label{eq:pseudonormlzero_largest_convex_spheres}
      \end{equation}
      that is, below the function
      \( \primal \in \RR^d \mapsto \inf \varphi\bset{ \LocalIndex \in \ic{0,d} }%
      {\primal \in \SupportDualNorm{\TripleNormSphere}{\LocalIndex}} \),
      with the convention that \( \SupportDualNorm{\TripleNormSphere}{0}=\{0\} \)
      and that \( \inf \emptyset = +\infty \).
    \end{subequations}
  \item \label{pr:eight}
    \begin{subequations}
      The proper convex lsc function~${\cal L}_{0}^{\varphi}$ also has three variational expressions as follows,
      where \( \Delta_{d+1} \) is the simplex of~$\RR^{d+1}$,
      \begin{align}
        {\cal L}_{0}^{\varphi}\np{\primal} 
        &= 
          \min_{ \substack{%
          \np{\lambda_0,\lambda_1,\ldots,\lambda_d} \in \Delta_{d+1} 
        \\
        \primal \in \sum_{ l=1 }^{ d } \lambda_\LocalIndex \SupportDualNorm{\TripleNormBall}{l} 
        } } 
        \sum_{ l=1}^{ d } \lambda_\LocalIndex \varphi\np{\LocalIndex}
        \eqsepv \forall \primal \in \RR^d 
        \label{eq:biconjugate_with_balls}
        \\
        &= 
          \min_{ \substack{%
          \np{\lambda_0,\lambda_1,\ldots,\lambda_d} \in \Delta_{d+1} 
        \\
        \primal \in \sum_{ l=1 }^{ d } \lambda_\LocalIndex \SupportDualNorm{\TripleNormSphere}{l} 
        } } 
        \sum_{ l=1 }^{ d } \lambda_\LocalIndex \varphi\np{\LocalIndex}
        \eqsepv \forall \primal \in \RR^d 
        \label{eq:biconjugate_with_spheres}
        \\
        &= 
          \min_{ \substack{%
          \primal^{(1)} \in \RR^d, \ldots, \primal^{(d)} \in \RR^d 
        \\
        \sum_{ \LocalIndex=1 }^{ d } \SupportDualNorm{\TripleNorm{\primal^{(\LocalIndex)}}}{\LocalIndex} \leq 1 
        \\
        \sum_{ \LocalIndex=1 }^{ d } \primal^{(\LocalIndex)} = \primal
        } }
        \sum_{ \LocalIndex=1 }^{ d } \varphi\np{\LocalIndex} \SupportDualNorm{\TripleNorm{\primal^{(\LocalIndex)}}}{\LocalIndex} 
        \eqsepv \forall \primal \in \RR^d 
        \eqfinp
        \label{eq:pseudonormlzero_convex_minimum}
      \end{align}
    \end{subequations}
  \end{enumerate}
  \label{pr:calL0}
\end{proposition}

\begin{proof}
  As in the beginning of the proof of
  Proposition~\ref{pr:nonempty_subdifferential},
  we make the observation that, 
  since the norm $\TripleNorm{\cdot}$ is orthant-strictly monotonic,
  it is orthant-monotonic, so that we have 
  \( \CoordinateNorm{\TripleNorm{\cdot}}{\LocalIndex} 
  =
  \SupportDualNorm{\TripleNorm{\cdot}}{\LocalIndex} \) 
  and
  \(      \CoordinateNormDual{\TripleNorm{\cdot}}{\LocalIndex}
  =
  \TopDualNorm{\TripleNorm{\cdot}}{\LocalIndex} \),
  for \( \LocalIndex\in\ic{0,d} \)
  by~\eqref{eq:dual_coordinate-k_norm_=_generalized_top-k_norm} in
  Proposition~\ref{pr:dual_coordinate-k_norm_=_generalized_top-k_norm}
  in Appendix~\ref{Coordinate-k_and_dual_coordinate-k_norms}
  (with the convention that these are the null seminorms in the case \( \LocalIndex=0 \)).
  \medskip

  \noindent\ref{pr:one} 
  As the Fenchel conjugacy induces a one-to-one correspondence
  between the closed convex functions on~$\RR^d$ and themselves
  \cite[Theorem~5]{Rockafellar:1974}, the function
  \( {\cal L}_{0}^{\varphi}=\LFMr{ \bp{ \SFM{\np{ \varphi \circ \lzero}}{\CouplingCapra} } } \) 
  in~\eqref{eq:definition_calL_0} is closed convex.
  We now show that it is proper.
  Indeed, on the one hand, it is easily seen,
  by the very definition~\eqref{eq:Fenchel-Moreau_conjugate}, that the function
  \( \SFM{\np{ \varphi \circ \lzero}}{\CouplingCapra} \) takes finite values,
  from which we deduce that the function
  \( \LFMr{ \bp{ \SFM{\np{ \varphi \circ \lzero}}{\CouplingCapra} } } \) 
  never takes the value~$-\infty$, by~\eqref{eq:Fenchel_conjugate_reverse}.
  On the other hand,
  we have \( \SFMbi{ \np{ \varphi \circ \lzero } }{\CouplingCapra} =
  \LFMr{ \bp{ \SFM{\np{ \varphi \circ \lzero}}{\CouplingCapra} } }
  \circ \normalized \) by 
  \cite[Equation~(30d)]{Chancelier-DeLara:2020_CAPRA_OPTIMIZATION},
  and \( \SFMbi{ \np{ \varphi \circ \lzero } }{\CouplingCapra}
  \leq \varphi \circ \lzero \) by~\eqref{eq:galois-cor},
  from which we deduce that, for any \( \primal \in \TripleNormSphere \),
  we have  
  \( \LFMr{ \bp{ \SFM{\np{ \varphi \circ \lzero}}{\CouplingCapra} } }\np{\primal}
  =\SFMbi{\np{ \varphi \circ \lzero }}{\CouplingCapra}\np{\primal}
  \leq \np{ \varphi \circ \lzero }\np{\primal} < +\infty \)
  since \( \varphi : \ic{0,d} \to \RR_+ \).
  As a consequence, the function \( \LFMr{ \bp{ \SFM{\np{ \varphi \circ \lzero}}{\CouplingCapra} } } \) 
  is proper.
  \medskip
  
  \noindent\ref{pr:two}
  The assumptions make it possible to conclude that 
  \( \SFMbi{ \np{ \varphi \circ \lzero } }{\CouplingCapra} = \varphi \circ \lzero \), 
  thanks to Theorem~\ref{th:pseudonormlzero_conjugate}.
  We deduce from 
  \cite[Proposition~4.3]{Chancelier-DeLara:2020_CAPRA_OPTIMIZATION} that,
  being $\CouplingCapra$-convex, the function~$\varphi \circ \lzero$ coincides, on the unit sphere~$\TripleNormSphere$, 
  with the closed convex function ${\cal L}_{0}^{\varphi} : \RR^d \to \barRR $
  given by~\cite[Equation~(30d)]{Chancelier-DeLara:2020_CAPRA_OPTIMIZATION}
  namely 
  \( {\cal L}_{0}^{\varphi} = \LFMr{ \bp{ \SFM{\np{ \varphi \circ \lzero
        }}{\CouplingCapra} } } \).
  Thus, we have proved~\eqref{eq:pseudonormlzero_convex_sphere}.
  \medskip

  \noindent\ref{pr:three} 
  Let \( \sphere \in \TripleNormSphere \). We
  prove~\eqref{eq:Capra-subdifferential=Rockafellar-Moreau-subdifferential}
  as follows:
  \begin{align*}
    \dual \in
    \subdifferential{}{{\cal L}_{0}^{\varphi}}\np{\sphere} 
    \iff &
           \LFM{ \np{ {\cal L}_{0}^{\varphi} } }\np{\dual}
           = \nscal{\sphere}{\dual} \LowPlus
           \bp{ -{\cal L}_{0}^{\varphi}\np{\sphere} }
           \intertext{by definition~\eqref{eq:Rockafellar-Moreau-subdifferential_a} of 
           the (Rockafellar-Moreau) subdifferential of a function}
    \iff &
           \LFM{ \Bp{ \LFMr{ \bp{ \SFM{\np{ \varphi \circ
           \lzero}}{\CouplingCapra} } } } }\np{\dual}
           = \nscal{\sphere}{\dual} \LowPlus
           \bgp{ -\Bp{ \LFMr{ \bp{ \SFM{\np{ \varphi \circ
           \lzero}}{\CouplingCapra} } } }\np{\sphere} }
           \tag{by definition~\eqref{eq:definition_calL_0} of
           \( {\cal L}_{0}^{\varphi}=\LFMr{ \bp{ \SFM{\np{ \varphi \circ \lzero}}{\CouplingCapra} } } \)}
    \\
    \iff &
           \SFM{\np{ \varphi \circ \lzero}}{\CouplingCapra}\np{\dual}
           = \nscal{\sphere}{\dual} \LowPlus
           \bgp{ -\Bp{ \LFMr{ \bp{ \SFM{\np{ \varphi \circ
           \lzero}}{\CouplingCapra} } } }\np{\sphere} }
           \intertext{because the function \( \SFM{\np{ \varphi \circ
           \lzero}}{\CouplingCapra} \) is a Fenchel conjugate
           by~\cite[Equation~(30b)]{Chancelier-DeLara:2020_CAPRA_OPTIMIZATION},
           hence is closed convex, hence is equal to its Fenchel biconjugate
           \( \LFM{ \Bp{ \LFMr{ \bp{ \SFM{\np{ \varphi \circ
           \lzero}}{\CouplingCapra} } } } } \)}
    \iff &
           \SFM{\np{ \varphi \circ \lzero}}{\CouplingCapra}\np{\dual}
           = \CouplingCapra\np{\sphere,\dual}  \LowPlus
           \bgp{ -\Bp{ \LFMr{ \bp{ \SFM{\np{ \varphi \circ
           \lzero}}{\CouplingCapra} } } }\np{\sphere} }
           \tag{by definition~\eqref{eq:coupling_CAPRA} of
           \( \CouplingCapra\np{\sphere,\dual} \) as \( \sphere \in \TripleNormSphere \) }
    \\
    \iff &
           \SFM{\np{ \varphi \circ \lzero}}{\CouplingCapra}\np{\dual}
           = \CouplingCapra\np{\sphere,\dual}  \LowPlus
           \Bp{ -\bp{ \SFMbi{\np{ \varphi \circ \lzero}}{\CouplingCapra}
           }\np{\sphere} }
           \intertext{because \( \SFMbi{ \np{ \varphi \circ \lzero } }{\CouplingCapra} =
           \LFMr{ \bp{ \SFM{\np{ \varphi \circ \lzero}}{\CouplingCapra} } }
           \circ \normalized \) by 
           \cite[Equation~(30d)]{Chancelier-DeLara:2020_CAPRA_OPTIMIZATION}, and using
           that \( \normalized\np{\sphere}=\sphere \) since 
           \( \sphere\in\TripleNormSphere \) by 
           definition~\eqref{eq:normalization_mapping} of the
           normalization mapping~$\normalized$}
    \iff &
           \SFM{\np{ \varphi \circ \lzero}}{\CouplingCapra}\np{\dual}
           = \CouplingCapra\np{\sphere,\dual}  \LowPlus
           \bp{ - \np{ \varphi \circ \lzero }\np{\sphere} }
           \tag{as \( \SFMbi{ \np{ \varphi \circ \lzero } }{\CouplingCapra} =
           \varphi \circ \lzero \)
           by Theorem~\ref{th:pseudonormlzero_conjugate}}
    \\
    \iff &
           \dual \in
           \subdifferential{\CouplingCapra}{\np{ \varphi \circ \lzero }}\np{\sphere}
           \eqfinp
           \tag{by definition~\eqref{eq:Capra-subdifferential_b} of
           the \capra-subdifferential}
  \end{align*}
  \medskip

  \noindent\ref{pr:four} 
  The equality~\eqref{eq:pseudonormlzero_convex_normalization_mapping}
  is a consequence of the formula
  \( \varphi \circ \lzero = \SFMbi{\np{ \varphi \circ \lzero}}{\CouplingCapra} =
  \LFMr{ \bp{ \SFM{\np{ \varphi \circ \lzero}}{\CouplingCapra} } } 
  \circ \normalized \)
  given by~\cite[Equation~(30d)]{Chancelier-DeLara:2020_CAPRA_OPTIMIZATION}.
  The equality~\eqref{eq:pseudonormlzero_convex_normalization_mapping_bis}
  is an easy consequence
  of~\eqref{eq:pseudonormlzero_convex_normalization_mapping}
  and of the definition~\eqref{eq:normalization_mapping}
  of the normalization mapping~$\normalized$.
  \medskip

  \noindent\ref{pr:five} 
  As \( {\cal L}_{0}^{\varphi} = \LFMr{ \bp{ \SFM{\np{ \varphi \circ \lzero }}{\CouplingCapra} } } \)
  by definition~\eqref{eq:definition_calL_0}, 
  and as we have that \( \LFMr{ \bp{ \SFM{\np{ \varphi \circ \lzero }}{\CouplingCapra} } }= \)
  \( \LFMr{ \Bp{ \sup_{\LocalIndex\in\ic{0,d}} \Bc{ \TopDualNorm{\TripleNorm{\cdot}}{\LocalIndex}-\varphi\np{\LocalIndex}
      } } } \) 
  by~\eqref{eq:conjugate_l0norm_TopDualNorm},
  we get~\eqref{eq:calL_0_definition}.
  \medskip

  \noindent\ref{pr:six} 
  We use Proposition~\ref{pr:pseudonormlzero_biconjugate_varphi}
  (Appendix~\ref{Appendix:Proposition})
  and especially Equations~\eqref{CAPRA-eq:caprastarprim-of-sup-norms} and~\eqref{CAPRA-eq:Biconjugate_of_min_balls_ind}
  to obtain~\eqref{eq:pseudonormlzero_largest_convex_balls}.
  Indeed, we have that 
  \begin{align}
    {\cal L}_{0}^{\varphi} 
    &= \LFMr{ \bp{ \SFM{\np{ \varphi \circ \lzero }}{\CouplingCapra} } }=
      \LFMr{ \Bp{ \sup_{\LocalIndex\in\ic{0,d}} \Bc{ \TopDualNorm{\TripleNorm{\cdot}}{\LocalIndex}-\varphi\np{\LocalIndex}
      } } }
      \tag{by~\eqref{eq:calL_0_definition} proved in Item~\ref{pr:five}}
    \\
    &=
      \LFMr{ 
      \Bp{ 
      \sup_{\LocalIndex\in\ic{0,d}} \bc{
      \CoordinateNormDual{\TripleNorm{\cdot}}{\LocalIndex} -\varphi\np{\LocalIndex} } } }
      \intertext{since $\TopDualNorm{\TripleNorm{\cdot}}{\LocalIndex}=
      \CoordinateNormDual{\TripleNorm{\cdot}}{\LocalIndex}$
      as recalled at the beginning of the proof}
    &=
      \LFMbi{ \Bp{ \inf_{\LocalIndex\in\ic{0,d}} \bc{ 
      \delta_{  \CoordinateNorm{\TripleNormBall}{\LocalIndex} } \UppPlus \varphi\np{\LocalIndex} } } }
      \tag{by~\eqref{CAPRA-eq:caprastarprim-of-sup-norms} and~\eqref{CAPRA-eq:Biconjugate_of_min_balls_ind}}
    \\
    &=
      \LFMbi{ \Bp{ \inf_{\LocalIndex\in\ic{0,d}} \bc{ 
      \delta_{ \SupportDualNorm{\TripleNormBall}{\LocalIndex} } \UppPlus \varphi\np{\LocalIndex} } } }
      \tag{as \( \CoordinateNorm{\TripleNormBall}{\LocalIndex}
      = \SupportDualNorm{\TripleNormBall}{\LocalIndex}\)
      since \( \CoordinateNorm{\TripleNorm{\cdot}}{\LocalIndex} 
      =
      \SupportDualNorm{\TripleNorm{\cdot}}{\LocalIndex} \),
      as recalled at the beginning of the proof}
      \eqfinv
  \end{align}
  which gives~\eqref{eq:pseudonormlzero_largest_convex_balls}.
  The inclusions and equality \(      \CoordinateNorm{\TripleNormBall}{1} 
  \subset \cdots \subset
  \CoordinateNorm{\TripleNormBall}{\LocalIndex} \subset \CoordinateNorm{\TripleNormBall}{\LocalIndex+1} 
  \subset \cdots \subset \CoordinateNorm{\TripleNormBall}{d} = \TripleNormBall \)
  have been established for the generalized coordinate-$k$ norms (see
  Definition~\ref{de:coordinate_norm}) in
  \cite[Equation~(24)]{Chancelier-DeLara:2020_CAPRA_OPTIMIZATION}. 
  Now, since  \( \CoordinateNorm{\TripleNorm{\cdot}}{\LocalIndex} 
  =
  \SupportDualNorm{\TripleNorm{\cdot}}{\LocalIndex} \),
  we get that 
  \( \SupportDualNorm{\TripleNormBall}{1}
  \subset \cdots \subset
  \SupportDualNorm{\TripleNormBall}{\LocalIndex-1} \subset \SupportDualNorm{\TripleNormBall}{\LocalIndex} 
  \subset \cdots \subset \SupportDualNorm{\TripleNormBall}{d} = \BALL \).
  \medskip

  \noindent\ref{pr:seven} 
  We use
  Proposition~\ref{pr:pseudonormlzero_biconjugate_varphi}
  (Appendix~\ref{Appendix:Proposition})
  and especially Equations~\eqref{CAPRA-eq:Biconjugate_of_min_balls_ind} 
  and~\eqref{CAPRA-eq:Biconjugate_of_min_spheres_ind},
  to obtain~\eqref{eq:pseudonormlzero_largest_convex_spheres}.
  Indeed, we have that 
  \begin{align}
    {\cal L}_{0}^{\varphi} 
    &=
      \LFMbi{ \Bp{ \inf_{\LocalIndex\in\ic{0,d}} \bc{ 
      \delta_{  \CoordinateNorm{\TripleNormBall}{\LocalIndex} } \UppPlus \varphi\np{\LocalIndex} } } }
      \tag{as seen in Item~\ref{pr:six}}
    \\
    &=
      \LFMbi{ \Bp{ \inf_{\LocalIndex\in\ic{0,d}} \bc{ 
      \delta_{ \CoordinateNorm{\TripleNormSphere}{\LocalIndex} } \UppPlus \varphi\np{\LocalIndex} } } }
      \tag{by~\eqref{CAPRA-eq:Biconjugate_of_min_balls_ind} 
      and~\eqref{CAPRA-eq:Biconjugate_of_min_spheres_ind}}
    \\
    &=
      \LFMbi{ \Bp{ \inf_{\LocalIndex\in\ic{0,d}} \bc{ 
      \delta_{\SupportDualNorm{\TripleNormSphere}{\LocalIndex}} \UppPlus \varphi\np{\LocalIndex} } } }
      \tag{since \( \CoordinateNorm{\TripleNormSphere}{\LocalIndex}
      = \SupportDualNorm{\TripleNormSphere}{\LocalIndex} \)
      as \( \CoordinateNorm{\TripleNorm{\cdot}}{\LocalIndex} 
      =
      \SupportDualNorm{\TripleNorm{\cdot}}{\LocalIndex} \)}
      \eqfinv
  \end{align}
  which gives~\eqref{eq:pseudonormlzero_largest_convex_spheres}.
  \medskip

  \noindent\ref{pr:eight} 
  We use the property that, for any \( k \in \ic{1,d} \), we have 
  \( \CoordinateNorm{\TripleNorm{\cdot}}{k} =
  \SupportDualNorm{\TripleNorm{\cdot}}{k} \) and also 
  Proposition~\ref{pr:pseudonormlzero_biconjugate_varphi}
  (Appendix~\ref{Appendix:Proposition})
  and especially Equations~\eqref{CAPRA-eq:biconjugate_with_balls},
  \eqref{CAPRA-eq:biconjugate_with_spheres}
  and \eqref{CAPRA-eq:pseudonormlzero_convex_minimum}
  to obtain~\eqref{eq:biconjugate_with_balls},
  \eqref{eq:biconjugate_with_spheres}
  and~\eqref{eq:pseudonormlzero_convex_minimum}.
  \medskip
  
  This ends the proof.
\end{proof}

\subsection{Variational formulation for the \lzeropseudonorm}
\label{Variational_formulation_for_the_pseudo_norm}

As an application of Proposition~\ref{pr:calL0},
we obtain the second main result of this paper, namely a variational formulation for
(nonnegative nondecreasing functions of) the \lzeropseudonorm.

\begin{theorem}
  Let $\TripleNorm{\cdot}$ be a norm on~$\RR^d$ with associated sequence
  \( \bseqa{\TopDualNorm{\TripleNorm{\cdot}}{\LocalIndex}}{\LocalIndex\in\ic{1,d}} \)
  of generalized $k$-support dual~norms
  as in Definition~\ref{de:top_dual_norm}, 
  and with
  associated \capra\ coupling~$\CouplingCapra$ as in~\eqref{eq:coupling_CAPRA}. 
  Suppose that both the norm $\TripleNorm{\cdot}$ and the dual norm $\TripleNormDual{\cdot}$
  are orthant-strictly monotonic.
  Let \( \varphi : \ic{0,d} \to \RR_+ \) be a nonnegative nondecreasing function
  such that \( \varphi\np{0}=0 \).
  Then, we have the equality 
  \begin{equation}
    \varphi\bp{ \lzero\np{\primal} }
    =
    \frac{ 1 }{ \TripleNorm{\primal} } 
    \min_{ \substack{%
        {\primal}^{(1)} \in \RR^d, \ldots, {\primal}^{(d)} \in \RR^d 
        \\
        \sum_{ \LocalIndex=1 }^{ d } \SupportDualNorm{\TripleNorm{{\primal}^{(\LocalIndex)}}}{\LocalIndex} \leq \TripleNorm{\primal}
        \\
        \sum_{ \LocalIndex=1 }^{ d } {\primal}^{(\LocalIndex)} = \primal  } }
    \sum_{ \LocalIndex=1 }^{ d } \varphi\np{\LocalIndex} \SupportDualNorm{\TripleNorm{{\primal}^{(\LocalIndex)}}}{\LocalIndex} 
    \eqsepv \forall \primal \in \RR^d\setminus\{0\}
    \eqfinv
    \label{eq:Variational_formulation_for_the_pseudo_norm}  
  \end{equation}
  where the sequence of 
  generalized $k$-support dual~norms 
  \( \bseqa{ \SupportDualNorm{\TripleNorm{\cdot}}{\LocalIndex} }{\LocalIndex\in\ic{1,d}} \) 
  has been introduced in Definition~\ref{de:top_dual_norm}.

  When \( \lzero\np{\primal} = l \geq 1 \), 
  the minimum in~\eqref{eq:Variational_formulation_for_the_pseudo_norm} 
  is achieved at 
  \( \np{ {\primal}^{(1)}, \ldots, {\primal}^{(d)} } \in \np{\RR^d}^d \) such that 
  \( {\primal}^{(j)}=0 \) for \( j \neq l \) and \( {\primal}^{(l)}=\primal \). 
  \label{th:Variational_formulation_for_the_pseudo_norm}
\end{theorem}

\begin{proof}
  Equation~\eqref{eq:Variational_formulation_for_the_pseudo_norm} 
  derives from~\eqref{eq:pseudonormlzero_convex_normalization_mapping}
  where we use the expression~\eqref{eq:pseudonormlzero_convex_minimum} 
  for the function \( {\cal L}_{0}^{\varphi} \) in~\eqref{eq:definition_calL_0}.

  Now for the argmin. 
  When \( \lzero\np{\primal} = l \geq 1 \), we have that
  \( \TripleNorm{\primal} = 
  \CoordinateNorm{\TripleNorm{\primal}}{d}
  = \cdots = 
  \CoordinateNorm{\TripleNorm{\primal}}{l} \)
  by \cite[Equation~(25a)]{Chancelier-DeLara:2020_CAPRA_OPTIMIZATION}.
  Now, for any \( k \in \ic{1,d} \), we have 
  \( \CoordinateNorm{\TripleNorm{\cdot}}{k} 
  =
  \SupportDualNorm{\TripleNorm{\cdot}}{k} \)
  by~\eqref{eq:dual_coordinate-k_norm_=_generalized_top-k_norm} in
  Proposition~\ref{pr:dual_coordinate-k_norm_=_generalized_top-k_norm}
  (Appendix~\ref{Coordinate-k_and_dual_coordinate-k_norms}),  
  since the norm $\TripleNorm{\cdot}$ is orthant-strictly monotonic,
  hence is orthant-monotonic.
  As a consequence, we have that 
  \( \TripleNorm{\primal} =
  \SupportDualNorm{\TripleNorm{\primal}}{d} 
  = \cdots = 
  \SupportDualNorm{\TripleNorm{\primal}}{l} \).
  Therefore, the vectors \( {\primal}^{(1)} \in \RR^d \), \ldots,
  \( {\primal}^{(d)} \in \RR^d \) defined by 
  \( {\primal}^{(j)}=0 \) for \( j \neq l \) and \( {\primal}^{(l)}=\primal \)
  are admissible for the minimization
  problem~\eqref{eq:Variational_formulation_for_the_pseudo_norm}.
  We deduce from~\eqref{eq:Variational_formulation_for_the_pseudo_norm} that 
  \( \varphi\np{l} = \varphi\bp{\lzero\np{\primal}} \leq 
  \frac{ 1 }{ \TripleNorm{\primal} } \varphi\np{l} \SupportDualNorm{\TripleNorm{\primal}}{l}
  =\varphi\np{l} \).
  \medskip

  This ends the proof. 
\end{proof}

As an illustration, Theorem~\ref{th:Variational_formulation_for_the_pseudo_norm} applies 
when the norm~$\TripleNorm{\cdot}$ is any of the 
$\ell_p$-norms~$\Norm{\cdot}_{p}$ on the space~\( \RR^d \), 
for $p\in ]1,\infty[$,
and Equation~\eqref{eq:Variational_formulation_for_the_pseudo_norm}
then gives (see the notations in the second column of
Table~\ref{tab:Examples}):
\( \forall \primal \in \RR^d\setminus\{0\}
\eqsepv \forall p\in ]1,\infty[ \), 
\begin{equation}
  \np{ \varphi \circ \lzero }\np{\primal} 
  =
  \frac{ 1 }{ \Norm{\primal}_{p} } 
  \min_{ \substack{%
      {\primal}^{(1)} \in \RR^d, \ldots, {\primal}^{(d)} \in \RR^d 
      \\
      \sum_{ \LocalIndex=1 }^{ d } \Norm{{\primal}^{(\LocalIndex)}}_{p,\LocalIndex}^{\mathrm{sn}} \leq \Norm{\primal}_{p}
      \\
      \sum_{ \LocalIndex=1 }^{ d } {\primal}^{(\LocalIndex)} = \primal  } }
  \sum_{ \LocalIndex=1 }^{ d } \varphi\np{\LocalIndex} \Norm{{\primal}^{(\LocalIndex)}}_{p,\LocalIndex}^{\mathrm{sn}} 
  \eqfinp 
  \label{eq:Variational_formulation_for_the_pseudo_norm_ell_p}
\end{equation}
Indeed, when $p\in ]1,\infty[$, 
the $\ell_p$-norm $\TripleNorm{\cdot}=\Norm{\cdot}_{p}$
is orthant-strictly monotonic, and so is its dual norm
$\TripleNormDual{\cdot}=\Norm{\cdot}_q$ where \( 1/p + 1/q = 1 \).
When $p=1$, the $\ell_1$-norm $\TripleNorm{\cdot}=\Norm{\cdot}_{1}$
is orthant-strictly monotonic, 
but the dual norm $\TripleNorm{\cdot}=\Norm{\cdot}_{\infty}$ is not;
when $p=\infty$, the $\ell_{\infty}$-norm $\TripleNorm{\cdot}=\Norm{\cdot}_{\infty}$
is not orthant-strictly monotonic;
hence, in those two extreme cases, we cannot conclude
(but we obtain inequalities like in
\cite[Equation~(25a)]{Chancelier-DeLara:2020_CAPRA_OPTIMIZATION}).

Finally, with the novel
expression~\eqref{eq:Variational_formulation_for_the_pseudo_norm}
for the \lzeropseudonorm, we deduce a possible reformulation of exact sparse
optimization problems as follows (the proof is a straightforward application of
Theorem~\ref{th:Variational_formulation_for_the_pseudo_norm}).

\begin{proposition}
  Let $\Convex \subset \RR^d$ be such that \( 0 \not\in \Convex \)
(if we had \( 0 \in \Convex \), the minimization problem below would
obviously be achieved at~\( \primal=0 \)).
  Let $\TripleNorm{\cdot}$ be a norm on~$\RR^d$, such that 
  both the norm $\TripleNorm{\cdot}$ and the dual norm $\TripleNormDual{\cdot}$
  are orthant-strictly monotonic.
  Let \( \varphi : \ic{0,d} \to \RR_+ \) be a nondecreasing function,
  such that \( \varphi\np{0}=0 \). 
  Then, we have that 
  \begin{subequations}
    \begin{align}
      \min_{ \primal \in \Convex } \varphi\bp{ \lzero\np{\primal} }
      &= 
        \min_{ \substack{%
        \primal \in \Convex, \primal^{(1)} \in \RR^d, \ldots, \primal^{(d)} \in \RR^d 
      \\
      \sum_{ \LocalIndex=1 }^{ d } \SupportDualNorm{\TripleNorm{\primal^{(\LocalIndex)}}}{\LocalIndex} \leq 1 
      \\
      \sum_{ \LocalIndex=1 }^{ d } \primal^{(\LocalIndex)} = \frac{ \primal }{ \TripleNorm{\primal} } } }
      \sum_{ \LocalIndex=1 }^{ d } \varphi\np{\LocalIndex} \SupportDualNorm{\TripleNorm{\primal^{(\LocalIndex)}}}{\LocalIndex} 
      \eqfinv
      \\
      &= 
        \min_{ \substack{%
        \primal \in \Convex, {\primal}^{(1)} \in \RR^d, \ldots, {\primal}^{(d)} \in \RR^d 
      \\
      \sum_{ \LocalIndex=1 }^{ d } \SupportDualNorm{\TripleNorm{{\primal}^{(\LocalIndex)}}}{\LocalIndex} \leq \TripleNorm{\primal}
      \\
      \sum_{ \LocalIndex=1 }^{ d } {\primal}^{(\LocalIndex)} = \primal  } }
      \frac{ 1 }{ \TripleNorm{\primal} } 
      \sum_{ \LocalIndex=1 }^{ d } \varphi\np{\LocalIndex} \SupportDualNorm{\TripleNorm{{\primal}^{(\LocalIndex)}}}{\LocalIndex} 
      \eqfinv
      \\
      &= 
        \min_{ \primal \in \Convex}
        \frac{ 1 }{ \TripleNorm{\primal} } 
        \underbrace{%
        \min_{ \substack{%
        {\primal}^{(1)} \in \RR^d, \ldots, {\primal}^{(d)} \in \RR^d 
      \\
      \sum_{ \LocalIndex=1 }^{ d } \SupportDualNorm{\TripleNorm{{\primal}^{(\LocalIndex)}}}{\LocalIndex} \leq \TripleNorm{\primal}
      \\
      \sum_{ \LocalIndex=1 }^{ d } {\primal}^{(\LocalIndex)} = \primal  } }
      \sum_{ \LocalIndex=1 }^{ d } \varphi\np{\LocalIndex} \SupportDualNorm{\TripleNorm{{\primal}^{(\LocalIndex)}}}{\LocalIndex} 
      }_{\textrm{convex optimization problem}}
      \eqfinp
    \end{align}
  \end{subequations}
  \label{pr:reformulation}
\end{proposition}

\section{Conclusion}

In this paper, we have proven that the \lzeropseudonorm\ is equal to its
\capra-biconjugate 
when both the source norm and its dual norm are orthant-strictly monotonic.
In that case, one says that the \lzeropseudonorm\
is a \capra-convex function. 
A surprising consequence is the convex factorization property,
a way to express hidden convexity:
the \lzeropseudonorm\ coincides, 
on the unit sphere of the source norm, with a proper convex lsc function.
More generally, this holds true for any function of the \lzeropseudonorm\ 
that is nondecreasing, with finite values.
Then, we have obtained exact variational formulations for the
\lzeropseudonorm, suitable for exact sparse optimization.
For this purpose, we have introduced
sequences of generalized top-$k$ and $k$-support dual~norms.
We now briefly sketch a few perspectives for exact sparse optimization.

The reformulations for exact sparse optimization problems,
obtained in Proposition~\ref{pr:reformulation},
have the nice feature to display partial convexity.
However, they make use of as many new (latent) vectors 
as the underlying dimension~$d$.
Thus, the algorithmic implementation may be delicate.
However, the variational formulation obtained
may suggest approximations of the \lzeropseudonorm,
involving generalized $k$-support dual~norms, which, themselves,
may lead to new smooth sparsity inducing terms.
Finally, we have identified elements of the \capra-subdifferential of nondecreasing
functions of the \lzeropseudonorm,
and we have related this \capra-subdifferential with the Rockafellar-Moreau subdifferential of the
associated convex lsc function (in the convex factorization property).
The identification of such subgradients could inspire ``gradient-like'' algorithms.
\bigskip

\textbf{Acknowledgements.}
We thank 
Guillaume Obozinski
for discussions on first versions of this work,
and Jean-Baptiste Hiriart-Urruty for his comments
(and for proposing the term ``convex factorization'').
We are indebted to two Reviewers and to the Editor who, by their questions and
comments, helped us improve the manuscript. 

\appendix

\section{Properties of relevant norms for the \lzeropseudonorm}
\label{Properties_of_relevant_norms_for_the_lzeropseudonorm}

We provide background on properties of norms that prove 
relevant for the \lzeropseudonorm.
In~\S\ref{Dual_norm_duality_normal_cone_exposed_face},
we review notions related to dual norms.
We establish properties of orthant-monotonic 
and orthant-strictly monotonic norms
in~\S\ref{Properties_of_Orthant-strictly_monotonic_norms},
and of coordinate-$k$ and dual coordinate-$k$ norms
in~\S\ref{Coordinate-k_and_dual_coordinate-k_norms}.

\subsection{Dual norm, \( \TripleNorm{\cdot} \)-duality, normal cone}
\label{Dual_norm_duality_normal_cone_exposed_face} 

For any norm~$\TripleNorm{\cdot}$ on~$\RR^d$, we recall that the following expression 
\begin{equation}
  \TripleNorm{\dual}_\star = 
  \sup_{ \TripleNorm{\primal} \leq 1 } \nscal{\primal}{\dual} 
  \eqsepv \forall \dual \in \RR^d
  \label{eq:dual_norm}
\end{equation}
defines a norm on~$\RR^d$, 
called the \emph{dual norm} \( \TripleNormDual{\cdot} \)
(in \cite[Section~15]{Rockafellar:1970}, this operation is widened to a polarity
operation between closed gauges).
\begin{subequations}
  By definition of the dual norm in~\eqref{eq:dual_norm}, we have 
  the inequality
  \begin{equation}
    \nscal{\primal}{\dual} \leq 
    \TripleNorm{\primal} \times \TripleNormDual{\dual} 
    \eqsepv \forall  \np{\primal,\dual} \in \RR^d \times \RR^d 
    \eqfinp 
    \label{eq:norm_dual_norm_inequality}
  \end{equation}
  We are interested in the case where this inequality is an equality.
  One says that \( \dual \in \RR^d \) is 
  \emph{\( \TripleNorm{\cdot} \)-dual} to \( \primal \in \RR^d \),
  denoted by \( \dual \parallel_{\TripleNorm{\cdot}} \primal \),
  if equality holds in inequality~\eqref{eq:norm_dual_norm_inequality},
  that is,
  \begin{equation}
    \dual \parallel_{\TripleNorm{\cdot}} \primal
    \iff
    \nscal{\primal}{\dual} =
    \TripleNorm{\primal} \times \TripleNormDual{\dual} 
    \eqfinp 
    \label{eq:couple_TripleNorm-dual}
  \end{equation}
\end{subequations}
The terminology \( \TripleNorm{\cdot} \)-dual comes from
\cite[page~2]{Marques_de_Sa-Sodupe:1993}
(see also the vocable of \emph{dual vector pair} in \cite[Equation~(1.11)]{Gries:1967}
and of \emph{dual vectors} in \cite[p.~283]{Gries-Stoer:1967}, 
whereas it is refered as \emph{polar alignment}
in~\cite{Fan-Jeong-Sun-Friedlander:2020}).
It will be convenient to express this notion of \( \TripleNorm{\cdot} \)-duality
in terms of geometric objects of convex analysis.
For this purpose,
we recall that the \emph{normal cone}~$\NORMAL_{\Convex}(\primal)$ 
to the (nonempty) closed convex subset~${\Convex} \subset \RR^d $
at~$\primal \in \Convex$ is the closed convex cone defined by 
\cite[p.136]{Hiriart-Urruty-Lemarechal-I:1993}
\begin{equation}
  \NORMAL_{\Convex}(\primal) =
  \bset{ \dual \in \RR^d}
  {
    \nscal{\dual}{\primal'-\primal} \leq 0 \eqsepv 
    \forall \primal' \in \Convex }
  \eqfinp
  \label{eq:normal_cone}
\end{equation}
Now, an easy computation shows that the notion
of \( \TripleNorm{\cdot} \)-duality can be rewritten in terms of
normal cone~\( \NORMAL_{ \TripleNormBall } \) 
as follows: 
\begin{equation}
  \dual \parallel_{\TripleNorm{\cdot}} \primal
  \iff
  \dual \in \NORMAL_{ \TripleNormBall }\bgp{ \frac{ \primal }{\TripleNorm{\primal} } }
  \eqsepv \forall  \np{\primal,\dual} 
  \in \bp{ \RR^d\setminus\{0\} } \times \RR^d 
  \eqfinp 
  \label{eq:couple_TripleNorm-dual_and_normal_cone} 
\end{equation}

\subsection{Properties of orthant-strictly monotonic norms}
\label{Properties_of_Orthant-strictly_monotonic_norms}

We provide useful properties of orthant-monotonic 
and orthant-strictly monotonic norms (see
Definition~\ref{de:orthant-monotonic}).
We recall that \( \primal_K \in \FlatRR_{K} \) denotes the vector which coincides with~\( \primal \),
except for the components outside of~$K$ that vanish, and that
the subspace~\( \FlatRR_{K} \) of~\( \RR^d \) has been defined
in~\eqref{eq:FlatRR}.

\begin{proposition}
  Let $\TripleNorm{\cdot}$ be an orthant-monotonic norm on~$\RR^d$.
  Then, the dual norm $\TripleNormDual{\cdot}$ is orthant-monotonic,
  and the norm $\TripleNorm{\cdot}$ is \emph{increasing with the coordinate subspaces},
  in the sense that, for any \( \primal \in \RR^d \)
  and any \( J \subset K \subset\ic{1,d} \),
  we have $ \TripleNorm{\primal_{J}} \leq \TripleNorm{\primal_K}$.
  \label{pr:orthant-monotonic}
\end{proposition}

\begin{proof}
  Let $\TripleNorm{\cdot}$ be an orthant-monotonic norm on~$\RR^d$.
  Then, by \cite[Theorem~2.23]{Gries:1967},
  the dual norm $\TripleNormDual{\cdot}$ is also orthant-monotonic and,
  by \cite[Proposition~2.4]{Marques_de_Sa-Sodupe:1993},
  we have that 
  \( \TripleNorm{u} \leq \TripleNorm{u + v} \),
for any subset \( J \subset \ic{1,d} \) and 
  for any vectors \( u \in \FlatRR_{J} \) and \( v \in \FlatRR_{-J} \)
  (following notation from game theory, 
  we have denoted by $-J$ the complementary subset of $J \subset \ic{1,d}$,
  that is, \( J \cup (-J) = \ic{1,d} \)     and \( J \cap (-J) = \emptyset \)).
  We consider \( \primal \in \RR^d \)
  and \( J \subset K \subset\ic{1,d} \).
  By setting \( u=\primal_{J} \in \FlatRR_{J} \)
  and \( v = \primal_K-\primal_{J} \), we get that 
  \( v \in \FlatRR_{-J} \), hence that $ \TripleNorm{\primal_{J}} \leq \TripleNorm{\primal_K}$.
\end{proof}

\begin{proposition}
  Let $\TripleNorm{\cdot}$ be an orthant-strictly monotonic norm on~$\RR^d$.
  Then
  \begin{enumerate}[label=({\bf\emph{\alph*}})]
  \item
    \label{it:SICS}
    the norm $\TripleNorm{\cdot}$ is \emph{strictly increasing
      with the coordinate subspaces} 
    in the sense that,
    for any \( \primal \in \RR^d \)
    and any \( J \subsetneq K \subset\ic{1,d} \),
    we have $ \primal_J \neq \primal_K \Rightarrow
    \TripleNorm{\primal_{J}} < \TripleNorm{\primal_K}$.
  \item
    \label{it:SDC}
    for any vector \( u \in \RR^d\setminus\{0\} \), 
    there exists a vector \( v \in \RR^d\setminus\{0\} \)
    such that \( \Support{v} = \Support{u} \),
    that \( u~\circ~v \geq 0 \), and that 
    $v$ is \( \TripleNorm{\cdot} \)-dual to~$u$, that is,
    \( \nscal{u}{v} = \TripleNorm{u} \times \TripleNormDual{v} \). 
  \end{enumerate}
  \label{pr:orthant-strictly_monotonic}
\end{proposition}

\begin{proof}
  \quad

  \ref{it:SICS}  Let \( \primal \in \RR^d \)
  and \( J \subsetneq K \subset\ic{1,d} \)
  be such that $ \primal_J \neq \primal_K $.
  We will show that \( \TripleNorm{\primal_K} > \TripleNorm{\primal_{J}} \).

  For this purpose, we set \( u=\primal_{J} \)
  and \( v = \primal_K-\primal_{J} \).
  Thus, we get that \( u \in \FlatRR_{K} \) and \( v \in
  \FlatRR_{-K}\setminus\{0\} \)
  (since \( J \subsetneq K \) and $ \primal_J \neq \primal_K $),
  that is, \( u=u_{K} \) and \( v=v_{-K}  \neq 0 \). 
  We are going to show that \( \TripleNorm{u + v} > \TripleNorm{u} \). 

  On the one hand, by definition of the module of a vector, 
  we easily see that 
  \( \module{w} = \module{ w_{K} } + \module{ w_{-K} } \),
  for any vector~\( w \in \RR^d \). 
  Thus, we have
  \( \module{u+v} = \module{ \np{u+v}_{K} } + \module{ \np{u+v}_{-K} } 
  = \module{ u_{K}+v_{K} } + \module{ u_{-K}+v_{-K} } 
  = \module{ u_{K}+0 } + \module{ 0+v_{-K} } 
  = \module{ u_{K} } + \module{v_{-K} } > \module{ u_{K} } =\module{ u } \) 
  since \( \module{v_{-K} } >0 \) as \( v=v_{-K} \neq 0 \),
  and since \( u=u_{K} \).  
  On the other hand, we easily get that \( \np{u+v}~\circ~u = 
  \bp{ \np{u+v}_{K}~\circ~u_{K} } + \bp{ \np{u+v}_{-K}~\circ~u_{-K} } = 
  \bp{ u_{K}~\circ~u_{K} } + \bp{ v_{-K}~\circ~u_{-K} } = 
  \bp{ u_{K}~\circ~u_{K} } \), because $u_{-K}=0$. 
  Therefore, we get that 
  \( \np{u+v}~\circ~u = \bp{ u_{K}~\circ~u_{K} } \geq 0 \).

  From \( \module{u+v} > \module{u} \) and 
  \( \np{u+v}~\circ~u \geq 0 \),
  we deduce that \( \TripleNorm{u + v} > \TripleNorm{u} \)
  by Definition~\ref{de:orthant-monotonic}
  as the norm \( \TripleNorm{\cdot} \) is orthant-strictly monotonic.
  Since \( u=\primal_{J} \) and \( v = \primal_K-\primal_{J} \),
  we conclude that
  \( \TripleNorm{\primal_{K}} > \TripleNorm{\primal_J} \).
  \medskip

  \ref{it:SDC}  Let \( u \in \RR^d\setminus\{0\} \) be given 
  and let us put \( K = \Support{u} \neq \emptyset \).
  As the norm \( \TripleNorm{\cdot} \) is orthant-strictly monotonic,
  it is orthant-monotonic; hence, 
  by
  \cite[Proposition~2.4]{Marques_de_Sa-Sodupe:1993},
  there exists a vector \( v \in \RR^d\setminus\{0\} \)
  such that \( \Support{v} \subset \Support{u} \),
  that \( u~\circ~v \geq 0 \) 
  and that $v$ is \( \TripleNorm{\cdot} \)-dual to~$u$,
  as in~\eqref{eq:couple_TripleNorm-dual},
  that is,  \( \nscal{u}{v} = \TripleNorm{u} \times \TripleNormDual{v} \). 
  Thus \( J = \Support{v} \subset K = \Support{u} \).
  We will now show that \( J \subsetneq K  \) is impossible,
  hence that \( J = K  \), thus proving that
  Item~\ref{it:SDC} holds true with the above vector~$v$.
  
  Writing that  \( \nscal{u}{v} = \TripleNorm{u} \times \TripleNormDual{v} \)
  (using that $u=u_K$ and $v=v_K=v_J$), we obtain
  \[
    \TripleNorm{u} \times \TripleNormDual{v} =
    \nscal{u}{v} = \nscal{u_K}{v} = \nscal{u_K}{v_K} = \nscal{u_K}{v_J} = 
    \nscal{u_J}{v_J} = \nscal{u_J}{v} 
    \eqfinv
  \]
  by obvious properties of the scalar product \( \nscal{\cdot}{\cdot} \).
  As a consequence, we get that 
  \( \{u_K,u_J \} \subset \argmax_{ \TripleNorm{\primal} \leq \TripleNorm{u} }
  \nscal{\primal}{v} \), by definition~\eqref{eq:dual_norm} of \(
  \TripleNormDual{v} \),
  because \( \TripleNorm{u}=\TripleNorm{u_K} \geq \TripleNorm{u_J} \),
  by Proposition~\ref{pr:orthant-monotonic}
  since \( J \subset K \) and the norm \( \TripleNorm{\cdot} \) 
  is orthant-monotonic. 
  But any solution in \( \argmax_{ \TripleNorm{\primal} \leq \TripleNorm{u} }
  \nscal{\primal}{v} \) belongs to the frontier of the ball of
  radius~$\TripleNorm{u}$,
  hence has exactly norm~$\TripleNorm{u}$. 
  Thus, we deduce that \( \TripleNorm{u}=\TripleNorm{u_K}= \TripleNorm{u_J} \).
  If we had \( J = \Support{v} \subsetneq K = \Support{u} \), 
  we would have $u_J \neq u_K$, hence 
  \( \TripleNorm{u_K} > \TripleNorm{u_J} \) by Item~\ref{it:SICS};
  this would be in contradiction with \( \TripleNorm{u_K} = \TripleNorm{u_J} \).
  Therefore,  \( J = \Support{v} =K= \Support{u} \). 
  \medskip

  This ends the proof.
\end{proof}

\subsection{Properties of coordinate-$k$ and dual coordinate-$k$ norms, and of generalized top-$k$ and $k$-support dual~norms}
\label{Coordinate-k_and_dual_coordinate-k_norms}

We establish useful properties of
coordinate-$k$ and dual coordinate-$k$ norms
(Definition~\ref{de:coordinate_norm}),
and of generalized top-$k$ and $k$-support dual~norms
(Definition~\ref{de:top_dual_norm}).

\begin{proposition}
  Let $\TripleNorm{\cdot}$ be a source norm on~$\RR^d$.

  Coordinate-$k$ norms are greater than $k$-support dual~norms,
  that is, 
  \begin{subequations}
    \begin{align}
      \CoordinateNorm{\TripleNorm{\primal}}{k} 
      & \geq 
        \SupportDualNorm{\TripleNorm{\primal}}{k}
        \eqsepv \forall \primal \in \RR^d 
        \eqsepv \forall k\in\ic{1,d}
        \eqfinv
        \label{eq:coordinate-k_norm_leq_generalized_k-support_norm}
        \intertext{whereas dual coordinate-$k$ norms
        are lower than generalized top-$k$ dual~norms, that is,}
        \CoordinateNormDual{\TripleNorm{\dual}}{k}
      & \leq 
        \TopDualNorm{\TripleNorm{\dual}}{k}
        \eqsepv \forall \dual \in \RR^d 
        \eqsepv \forall k\in\ic{1,d}
        \eqfinp
        \label{eq:dual_coordinate-k_norm_geq_generalized_top-k_norm}
    \end{align}
  \end{subequations}
  If the source norm norm $\TripleNorm{\cdot}$ is orthant-monotonic,
  then equalities hold true, that is, 
  \begin{equation}
    \TripleNorm{\cdot} \textrm{is orthant-monotonic}
    \Rightarrow
    \forall k\in\ic{1,d} \quad
    \begin{cases}
      \CoordinateNorm{\TripleNorm{\cdot}}{k} 
      \! &= 
      \SupportDualNorm{\TripleNorm{\cdot}}{k}
      \eqfinv 
      \\
      \CoordinateNormDual{\TripleNorm{\cdot}}{k}
      \! &= 
      \TopDualNorm{\TripleNorm{\cdot}}{k}
      \eqfinp 
    \end{cases}
    \label{eq:dual_coordinate-k_norm_=_generalized_top-k_norm}
  \end{equation}
  \label{pr:dual_coordinate-k_norm_=_generalized_top-k_norm}
\end{proposition}

\begin{proof}
  It is known that, for any nonempty subset \( K \subset \ic{1,d} \), 
  we have the inequality \( \TripleNorm{\cdot}_{K,\star} 
  \leq \TripleNorm{\cdot}_{\star,K} \)
  (see \cite[Proposition~2.2]{Marques_de_Sa-Sodupe:1993}).
  From the definition~\eqref{eq:top_dual_norm} of the generalized top-$k$
  dual~norm,
  and the definition~\eqref{eq:dual_coordinate_norm_definition} 
  of the dual coordinate-$k$ norm,
  we get that 
  \(   \CoordinateNormDual{\TripleNorm{\dual}}{k}
  = \sup_{\cardinal{K} \leq k} \TripleNorm{\dual_K}_{K,\star}
  \leq
  \sup_{\cardinal{K} \leq k} \TripleNorm{\dual_K}_{\star,K}
  = \TopDualNorm{\TripleNorm{\dual}}{k}
  \), hence we obtain~\eqref{eq:dual_coordinate-k_norm_geq_generalized_top-k_norm}.
  By taking the dual norms, we
  get~\eqref{eq:coordinate-k_norm_leq_generalized_k-support_norm}.

  The norms for which the equality 
  \( \TripleNorm{\cdot}_{K,\star} =
  \TripleNorm{\cdot}_{\star,K} \)
  holds true for all nonempty subsets 
  \( K \subset\ic{1,d} \),
  are the orthant-monotonic norms (\cite[Characterization~2.26]{Gries:1967}, %
  \cite[Theorem~3.2]{Marques_de_Sa-Sodupe:1993}).
  Therefore, if the norm~$\TripleNorm{\cdot}$ is orthant-monotonic,
  from the definition~\eqref{eq:top_dual_norm} of the generalized top-$k$
  dual~norm, we get that the inequality~\eqref{eq:dual_coordinate-k_norm_geq_generalized_top-k_norm}
  becomes an equality. Then, the
  inequality~\eqref{eq:coordinate-k_norm_leq_generalized_k-support_norm}
  also becomes an equality by taking the dual norm as in~\eqref{eq:dual_norm}.
  Thus, we have
  obtained~\eqref{eq:dual_coordinate-k_norm_=_generalized_top-k_norm}. 

  This ends the proof.
\end{proof}

\begin{proposition}
  Let $\TripleNorm{\cdot}$ be a source norm on~$\RR^d$.
  Let \( \dual\in\RR^d \) and \( l\in\ic{1,d} \).
  If the dual norm $\TripleNormDual{\cdot}$ is orthant-strictly monotonic,
  we have that
  \begin{equation}
    \lzero\np{\dual} = l \implies
    \begin{cases}
      \CoordinateNormDual{\TripleNorm{\dual}}{1} < \cdots < 
      \CoordinateNormDual{\TripleNorm{\dual}}{l-1}  < 
      \CoordinateNormDual{\TripleNorm{\dual}}{l} =
      \cdots = \CoordinateNormDual{\TripleNorm{\dual}}{d} =\TripleNormDual{\dual}
      \eqfinv
      \\
      \TopDualNorm{\TripleNorm{\dual}}{1} < \cdots < 
      \TopDualNorm{\TripleNorm{\dual}}{l-1}  < 
      \TopDualNorm{\TripleNorm{\dual}}{l} =
      \cdots = \TopDualNorm{\TripleNorm{\dual}}{d} =\TripleNormDual{\dual}
      \eqfinp
    \end{cases}
    \label{eq:level_curve_l0_characterization}
  \end{equation}
  \label{pr:assumptions_pseudonormlzero_conjugate}
\end{proposition}

\begin{proof}
  We consider \( \dual\in\RR^d \).
  We put \( L=\Support{\dual} \) and we suppose that 
  $\lzero(\dual)=\cardinal{L} = l $.
  
  Since the norm $\TripleNormDual{\cdot}$ is orthant-strictly monotonic,
  it is orthant-monotonic and so is $\TripleNorm{\cdot}$
  by Proposition~\ref{pr:orthant-monotonic}.
  By~\eqref{eq:dual_coordinate-k_norm_=_generalized_top-k_norm} in
  Proposition~\ref{pr:dual_coordinate-k_norm_=_generalized_top-k_norm}, we get that
  \( \CoordinateNorm{\TripleNorm{\cdot}}{\LocalIndex} 
  =
  \SupportDualNorm{\TripleNorm{\cdot}}{\LocalIndex} \) 
  and
  \(      \CoordinateNormDual{\TripleNorm{\cdot}}{\LocalIndex}
  =
  \TopDualNorm{\TripleNorm{\cdot}}{\LocalIndex} \),
  for \( \LocalIndex\in\ic{0,d} \) 
  (with the convention that these are the null seminorms in the case~\( \LocalIndex=0 \)).
  Therefore, we can translate all the results, obtained in
  \cite{Chancelier-DeLara:2020_CAPRA_OPTIMIZATION},
  with coordinate-$k$ and dual coordinate-$k$ norms,
  into results regarding
  generalized top-$k$ and $k$-support dual~norms.
  As an application, by \cite[Equation~(18)]{Chancelier-DeLara:2020_CAPRA_OPTIMIZATION},
  we get, from $\lzero(\dual)=l$, that
  \begin{equation}
    \TopDualNorm{\TripleNorm{\dual}}{1} \leq \cdots \leq 
    \TopDualNorm{\TripleNorm{\dual}}{\LocalIndex} \leq
    \TopDualNorm{\TripleNorm{\dual}}{\LocalIndex+1} \leq \cdots \leq 
    \TopDualNorm{\TripleNorm{\dual}}{d} =\TripleNormDual{\dual}
    \eqsepv \forall \dual \in \RR^d
    \eqfinp
    \label{eq:top_dual_norm_inequalities}
  \end{equation}
  We now prove~\eqref{eq:level_curve_l0_characterization} in two steps. 
  \medskip

  We first show that \( \TopDualNorm{\TripleNorm{\dual}}{l} =
  \cdots = \TopDualNorm{\TripleNorm{\dual}}{d} =\TripleNormDual{\dual} \)
  (the right hand side of~\eqref{eq:level_curve_l0_characterization}).
  Since \( \dual=\dual_L \), by definition of the set~$L=\Support{\dual}$, we
  have that 
  \(
  \TripleNormDual{\dual} = \TripleNormDual{\dual_L} \leq
  \sup_{\cardinal{K} \leq l} \TripleNormDual{\dual_K} = \TopDualNorm{\TripleNorm{\dual}}{l}
  \)
  by the very definition~\eqref{eq:top_dual_norm}
  of the generalized top-$l$ dual~norm \( \TopDualNorm{\TripleNorm{\cdot}}{l} \).
  By~\eqref{eq:top_dual_norm_inequalities}, we conclude that
  \( \TopDualNorm{\TripleNorm{\dual}}{l} =
  \cdots = \TopDualNorm{\TripleNorm{\dual}}{d} =\TripleNormDual{\dual} \).
  \medskip

  Second, we show that 
  \( \TopDualNorm{\TripleNorm{\dual}}{1} < \cdots < 
  \TopDualNorm{\TripleNorm{\dual}}{l-1} < \TopDualNorm{\TripleNorm{\dual}}{l} \)
  (the left hand side of~\eqref{eq:level_curve_l0_characterization}).
  There is nothing to show for~$l=0$.
  Now, for \( l \geq 1 \) and for any \( k \in \ic{0,l-1} \), we have
  \begin{align*}
    \TopDualNorm{\TripleNorm{\dual}}{k}  
    &=
      \sup_{\cardinal{K} \leq k} 
      \TripleNormDual{\dual_{K}}
      \tag{by definition~\eqref{eq:top_dual_norm} of the generalized top-$k$ dual~norm } 
    \\
    &=
      \sup_{\cardinal{K} \leq k} 
      \TripleNormDual{\dual_{K \cap L}}
      \tag{because \( \dual_L=\dual \) by definition of the set~$L=\Support{\dual}$ }
    \\
    &=
      \sup_{\cardinal{K'} \leq k, K' \subset L} 
      \TripleNormDual{\dual_{K'}}
      \tag{by setting \( K'=K \cap L \) }
    \\
    &=
      \sup_{\cardinal{K} \leq k, K \subset L} 
      \TripleNormDual{\dual_{K}}
      \tag{the same but with notation $K$ instead of $K'$}
    \\
    &=
      \sup_{\cardinal{K} \leq k, K \subsetneq L } 
      \TripleNormDual{\dual_{K}}
      \tag{because \( \cardinal{K} \leq k \leq l-1 < l =\cardinal{L} \)
      implies that \( K \neq L \)}
    \\
    &< 
      \sup_{\substack{\cardinal{K} \leq k, \LocalIndex \in L\setminus K \\ K \subsetneq L} } 
    \TripleNormDual{\dual_{K \cup \na{\LocalIndex}}}
    \intertext{%
    because the set \( L \setminus K \) is nonempty
    (having cardinality 
    \( \cardinal{L}-\cardinal{K}=l-\cardinal{K} \geq k+1-\cardinal{K} \geq 1 \)),
    and because, since the norm $\TripleNormDual{\cdot}$ is orthant-strictly monotonic, 
    using Item~\ref{it:SICS} in Proposition~\ref{pr:orthant-strictly_monotonic},
    we obtain that \( \TripleNormDual{\dual_{K}} < \TripleNormDual{\dual_{K \cup \na{\LocalIndex}}} \) 
    as \( \dual_{K} \neq \dual_{K \cup \na{\LocalIndex}} \) for at least one
    \( \LocalIndex \in L\setminus K \) since  \( L=\Support{\dual} \)}
    & \leq
      \sup_{\cardinal{J} \leq k+1, J \subset L } 
      \TripleNormDual{\dual_{J}}
      \tag{as all the subsets $K'=K \cup \na{\LocalIndex}$ are such that $K' \subset L $ and 
      \( \cardinal{K'}=k+1 \)}
    \\
    & \leq 
      \TopDualNorm{\TripleNorm{\dual}}{k+1}  
  \end{align*}
  by definition~\eqref{eq:top_dual_norm} of the generalized top-($k+1$) dual~norm
  (in fact the last inequality is easily shown to be an equality as
  \( \dual_L=\dual \)).
  Thus, for any \( k \in \ic{0,l-1} \), we have established that 
  \( \TopDualNorm{\TripleNorm{\dual}}{k} < \TopDualNorm{\TripleNorm{\dual}}{k+1} \). 
  \medskip

  This ends the proof.
\end{proof}

\section{Proposition~\ref{pr:pseudonormlzero_biconjugate_varphi}} 
\label{Appendix:Proposition}

We reproduce here~\cite[Proposition~4.5]{Chancelier-DeLara:2020_CAPRA_OPTIMIZATION}
in order to simplify the reading of the proof of Proposition~\ref{pr:calL0}.

\begin{proposition}(\cite[Proposition~4.5]{Chancelier-DeLara:2020_CAPRA_OPTIMIZATION})
  Let $\TripleNorm{\cdot}$ be a norm on~$\RR^d$,
  with associated sequence
  \( \sequence{\CoordinateNorm{\TripleNorm{\cdot}}{\LocalIndex}}{\LocalIndex\in\ic{1,d}} \)
  of coordinate-$k$ norms and sequence
  \( \sequence{\CoordinateNormDual{\TripleNorm{\cdot}}{\LocalIndex}}{\LocalIndex\in\ic{1,d}} \)
  of dual coordinate-$k$ norms, as 
  in Definition~\ref{de:coordinate_norm},
  and with
  associated \Capra\ coupling $\CouplingCapra$ in~\eqref{eq:coupling_CAPRA}. 

  \begin{enumerate}
  \item 
    For any function \( \varphi : \ic{0,d} \to \barRR \), we have
    \begin{subequations}
      \begin{align}
        \SFMbi{ \np{ \varphi \circ \lzero} }{\CouplingCapra}\np{\primal}
        &=
          \LFMr{ \bp{ \SFM{\np{ \varphi \circ \lzero}}{\CouplingCapra} } }
          \np{ \frac{\primal}{\TripleNorm{\primal}} }
          \eqsepv \forall \primal \in \RR^d\backslash\{0\}
          \label{CAPRA-eq:Biconjugate_ofvarphi_of_lzero}
          \eqfinv 
          \intertext{where the closed convex function 
          \( \LFMr{ \bp{ \SFM{\np{ \varphi \circ \lzero}}{\CouplingCapra} } } \) 
          has the following expression as a Fenchel conjugate}
          \LFMr{ \bp{ \SFM{\np{ \varphi \circ \lzero}}{\CouplingCapra} } }
        &=
          \LFMr{ \Bp{ 
          \sup_{\LocalIndex\in\ic{0,d}} \bc{
          \CoordinateNormDual{\TripleNorm{\cdot}}{\LocalIndex} -\varphi\np{\LocalIndex} } } }
          \eqfinv 
          \label{CAPRA-eq:caprastarprim-of-sup-norms}
          \intertext{and also
          has the following four expressions as a Fenchel biconjugate}
        &=
          \LFMbi{ \Bp{ \inf_{\LocalIndex\in\ic{0,d}} \bc{ 
          \delta_{  \CoordinateNorm{\TripleNormBall}{\LocalIndex} } \UppPlus \varphi\np{\LocalIndex} } } }
          \eqfinv 
          \label{CAPRA-eq:Biconjugate_of_min_balls_ind}
          \intertext{hence the function~\( \LFMr{ \bp{ \SFM{\np{ \varphi \circ \lzero}}{\CouplingCapra} } } \)
          is the largest closed convex function
          below the integer valued function 
          \( \inf_{\LocalIndex\in\ic{0,d}} \bc{ 
          \delta_{  \CoordinateNorm{\TripleNormBall}{\LocalIndex} } \UppPlus \varphi\np{\LocalIndex} } \),
          such that 
          \( \primal \in \CoordinateNorm{\TripleNormBall}{\LocalIndex} 
          \backslash \CoordinateNorm{\TripleNormBall}{\LocalIndex-1} \mapsto \varphi\np{\LocalIndex} \) 
          for $l\in\ic{1,d}$,
          and $\primal \in \CoordinateNorm{\TripleNormBall}{0} = \{0\} \mapsto \varphi\np{0}$, the function
          being infinite outside~\( \CoordinateNorm{\TripleNormBall}{d}=
          \TripleNormBall \), that is, with the convention that \( \CoordinateNorm{\TripleNormBall}{0}=\{0\} \)
          and that \( \inf \emptyset = +\infty \)}
        &= 
          \LFMbi{ \Bp{ \primal \mapsto \inf \bset{ \varphi\np{\LocalIndex} }%
          { \primal \in \CoordinateNorm{\TripleNormBall}{\LocalIndex}
          \eqsepv \LocalIndex \in \ic{0,d} } } }
          \eqfinv 
          \label{CAPRA-eq:Biconjugate_of_min_balls_ind_bis}
        \\
        &=
          \LFMbi{ \Bp{ \inf_{\LocalIndex\in\ic{0,d}} \bc{ 
          \delta_{ \CoordinateNorm{\TripleNormSphere}{\LocalIndex} } \UppPlus \varphi\np{\LocalIndex} } } }
          \eqfinv 
          \label{CAPRA-eq:Biconjugate_of_min_spheres_ind}
          \intertext{hence the function~\( \LFMr{ \bp{ \SFM{\np{ \varphi \circ \lzero}}{\CouplingCapra} } } \)
          is the largest closed convex function
          below the integer valued function 
          \( \inf_{\LocalIndex\in\ic{0,d}} \bc{ 
          \delta_{  \CoordinateNorm{\TripleNormSphere}{\LocalIndex} } \UppPlus
          \varphi\np{\LocalIndex} } \), that is,
          with the convention that \( \CoordinateNorm{\TripleNormSphere}{0}=\{0\} \)
          and that \( \inf \emptyset = +\infty \)}
        &= 
          \LFMbi{ \Bp{\primal \mapsto \inf \bset{ \varphi\np{\LocalIndex} }%
          { \primal \in \CoordinateNorm{\TripleNormSphere}{\LocalIndex}
          \eqsepv \LocalIndex \in \ic{0,d} } } }
          \eqfinp 
          \label{CAPRA-eq:Biconjugate_of_min_spheres_ind_bis}
      \end{align}
    \item 
      For any function \( \varphi : \ic{0,d} \to \RR \), 
      that is, with finite values, the function 
      \( \LFMr{ \bp{ \SFM{\np{ \varphi \circ \lzero}}{\CouplingCapra} } } \) 
      is proper convex lsc and 
      has the following variational expression
      (where \( \Delta_{d+1} \) denotes the simplex of~$\RR^{d+1}$)
      \begin{align}
        \LFMr{ \bp{ \SFM{\np{ \varphi \circ \lzero}}{\CouplingCapra} } }\np{\primal}
        &= 
          \min_{ \substack{%
          \np{\lambda_0,\lambda_1,\ldots,\lambda_d} \in \Delta_{d+1} 
        \\
        \primal \in \sum_{ \LocalIndex=1 }^{ d } \lambda_{\LocalIndex} \CoordinateNorm{\TripleNormBall}{\LocalIndex} 
        } } 
        \sum_{ \LocalIndex=0}^{ d } \lambda_{\LocalIndex}
        \varphi\np{\LocalIndex} 
        \eqsepv \forall  \primal \in \RR^d 
        \eqfinp
        \label{CAPRA-eq:biconjugate_with_balls}
      \end{align}
    \item
      \label{it:biconjugate_l0norm_varphi}
      For any function \( \varphi : \ic{0,d} \to \RR_+ \), 
      that is, with nonnegative finite values,
      and such that \( \varphi\np{0}=0 \), the function 
      \( \LFMr{ \bp{ \SFM{\np{ \varphi \circ \lzero}}{\CouplingCapra} } } \) 
      is proper convex lsc and 
      has the following two variational expressions
      (notice that, in~\eqref{CAPRA-eq:biconjugate_with_balls}, the sum starts from \( \LocalIndex=0 \),
        whereas in~\eqref{CAPRA-eq:biconjugate_with_spheres}
        and in~\eqref{CAPRA-eq:pseudonormlzero_convex_minimum}, 
        the sum starts from \( \LocalIndex=1 \)) 
      \begin{align}
        \LFMr{ \bp{ \SFM{\np{ \varphi \circ \lzero}}{\CouplingCapra} } }\np{\primal}
        &= 
          \min_{ \substack{%
          \np{\lambda_0,\lambda_1,\ldots,\lambda_d} \in \Delta_{d+1} 
        \\
        \primal \in \sum_{ \LocalIndex=1 }^{ d } \lambda_{\LocalIndex} \CoordinateNorm{\TripleNormSphere}{\LocalIndex} 
        } } 
        \sum_{ \LocalIndex=1 }^{ d } \lambda_{\LocalIndex} \varphi\np{\LocalIndex} 
        \eqsepv \forall \primal \in \RR^d
        \eqfinv
        \label{CAPRA-eq:biconjugate_with_spheres}
        \\
        &= 
          \min_{ \substack{%
          z^{(1)} \in \RR^d, \ldots, z^{(d)} \in \RR^d 
        \\
        \sum_{ \LocalIndex=1 }^{ d } \CoordinateNorm{\TripleNorm{z^{(\LocalIndex)}}}{\LocalIndex} \leq 1
        \\
        \sum_{ \LocalIndex=1 }^{ d } z^{(\LocalIndex)} = \primal } }
        \sum_{ \LocalIndex=1 }^{ d } \varphi\np{\LocalIndex}
        \CoordinateNorm{\TripleNorm{z^{(\LocalIndex)}}}{\LocalIndex} 
        \eqsepv \forall \primal \in \RR^d
        \eqfinv
        \label{CAPRA-eq:pseudonormlzero_convex_minimum}
      \end{align}
      %
    \end{subequations}
    and the function \( \SFMbi{ \np{ \varphi \circ \lzero} }{\CouplingCapra} \)
    has the following variational expression
    \begin{equation}
      \SFMbi{ \np{ \varphi \circ \lzero} }{\CouplingCapra}\np{\primal}
      =
      \frac{ 1 }{ \TripleNorm{\primal} } 
      \min_{ \substack{%
          z^{(1)} \in \RR^d, \ldots, z^{(d)} \in \RR^d 
          \\
          \sum_{ \LocalIndex=1 }^{ d } \CoordinateNorm{\TripleNorm{z^{(\LocalIndex)}}}{\LocalIndex} \leq \TripleNorm{\primal}
          \\
          \sum_{ \LocalIndex=1 }^{ d } z^{(\LocalIndex)} = \primal } }
      \sum_{ \LocalIndex=1 }^{ d }
      \CoordinateNorm{\TripleNorm{z^{(\LocalIndex)}}}{\LocalIndex} \varphi\np{\LocalIndex} 
      \eqsepv \forall \primal \in \RR^d\backslash\{0\} 
      \eqfinp
      \label{CAPRA-eq:biconjugate_l0norm_varphi}
    \end{equation}
  \end{enumerate}
  \label{pr:pseudonormlzero_biconjugate_varphi}
\end{proposition}

\section{Background on the Fenchel conjugacy on~$\RR^d$}
\label{The_Fenchel_conjugacy}

We review concepts and notations related to the Fenchel conjugacy
(we refer the reader to \cite{Rockafellar:1974}).
For any function \( \fonctionuncertain : \RR^d \to \barRR \),
its \emph{epigraph} is \( \epigraph\fonctionuncertain= 
\defset{ \np{\uncertain,t}\in\RR^d\times\RR}%
{\fonctionuncertain\np{\uncertain} \leq t} \),
its \emph{effective domain} is 
\( \dom\fonctionuncertain= 
\defset{\uncertain\in\RR^d}{ \fonctionuncertain\np{\uncertain} <+\infty}
\).
A function \( \fonctionuncertain : \RR^d \to \barRR \)
is said to be
\emph{convex} if its epigraph is a convex set,
\emph{proper} if it never takes the value~$-\infty$
and that \( \dom\fonctionuncertain \not = \emptyset \),
\emph{lower semi continuous (lsc)} if its epigraph is closed,
\emph{closed} if it either lsc and nowhere having the value $-\infty$,
or is the constant function~$-\infty$ \cite[p.~15]{Rockafellar:1974}.
Closed convex functions are the two constant functions~$-\infty$ and~$+\infty$
united with all proper convex lsc functions.
  In particular, any closed convex function that takes at least one finite value
  is necessarily proper convex~lsc. 

For any functions \( \fonctionprimal : \RR^d  \to \barRR \)
and \( \fonctiondual : \RR^d \to \barRR \), 
we denote
\begin{subequations}
  \begin{align}
    \LFM{\fonctionprimal}\np{\dual} 
    &= 
      \sup_{\primal \in \RR^d} \Bp{ \nscal{\primal}{\dual} 
      \LowPlus \bp{ -\fonctionprimal\np{\primal} } } 
      \eqsepv \forall \dual \in \RR^d
      \eqfinv
      \label{eq:Fenchel_conjugate}
    \\
    \LFMr{\fonctiondual}\np{\primal} 
    &= 
      \sup_{ \dual \in \DUAL } \Bp{ \nscal{\primal}{\dual} 
      \LowPlus \bp{ -\fonctiondual\np{\dual} } } 
      \eqsepv \forall \primal \in \RR^d
      \eqfinv
      \label{eq:Fenchel_conjugate_reverse}
    \\
    \LFMbi{\fonctionprimal}\np{\primal} 
    &= 
      \sup_{\dual \in \RR^d} \Bp{ \nscal{\primal}{\dual} 
      \LowPlus \bp{ -\LFM{\fonctionprimal}\np{\dual} } } 
      \eqsepv \forall \primal \in \RR^d
      \eqfinp
      \label{eq:Fenchel_biconjugate}
  \end{align}
\end{subequations}
  In convex analysis, one does not use
  the notation~\( \LFMr{} \)
  in~\eqref{eq:Fenchel_conjugate_reverse} and~\( \LFMbi{} \) in~\eqref{eq:Fenchel_biconjugate},
  but simply~\( \LFM{} \) and~\( ^{\star\star} \).
  We use~\( \LFMr{} \) and~\( \LFMbi{} \) to be consistent with the
  notation~\eqref{eq:Fenchel-Moreau_biconjugate} for general conjugacies.

It is proved that the Fenchel conjugacy
(indifferently 
\( \fonctionprimal \mapsto \LFM{\fonctionprimal} \)
or
\( \fonctiondual \mapsto \LFMr{\fonctiondual} \))
induces a one-to-one correspondence
between the closed convex functions on~$\RR^d$ and themselves
\cite[Theorem~5]{Rockafellar:1974}.

In \cite[p.~214-215]{Rockafellar:1970}
(see also the historical note in \cite[p.~343]{Rockafellar-Wets:1998}),
the notions of (Moreau) subgradient
and of (Rockafellar) subdifferential are defined for a convex function.
Following the definition of the subdifferential
of a function with respect to a duality 
in \cite{Akian-Gaubert-Kolokoltsov:2002},
we define the \emph{(Rockafellar-Moreau) subdifferential} \( \subdifferential{}{\fonctionprimal}\np{\primal} \)
of a function \( \fonctionprimal : \RR^d  \to \barRR \)
at~\( \primal \in \RR^d \) by
\begin{subequations}
  \begin{equation}
    \subdifferential{}{\fonctionprimal}\np{\primal}
    =    \defset{ \dual \in \RR^d }{ %
      \LFM{\fonctionprimal}\np{\dual} 
      = \nscal{\primal}{\dual} 
      \LowPlus \bp{ -\fonctionprimal\np{\primal} } }
    \eqfinp
    \label{eq:Rockafellar-Moreau-subdifferential_a}
  \end{equation}
  When the function~\( \fonctionprimal \) is proper convex and
  \( \primal \in  \dom\fonctionprimal \), we recover the classic definition that
  \begin{equation}
    \subdifferential{}{\fonctionprimal}\np{\primal}
    =    \defset{ \dual \in \RR^d }{ %
      \nscal{\primal'-\primal}{\dual}+\fonctionprimal\np{\primal}
      \leq \fonctionprimal\np{\primal'}
      \eqsepv \forall \primal' \in \dom\fonctionprimal }
    \eqfinp
    \label{eq:Rockafellar-Moreau-subdifferential_b}
  \end{equation}  
\end{subequations}

\newcommand{\noopsort}[1]{} \ifx\undefined\allcaps\def\allcaps#1{#1}\fi

\end{document}